\documentclass[reqno,11pt,oneside,final]{amsart}

\usepackage{amsmath,amssymb,amsthm,bbm,marvosym}
\usepackage{epsfig,psfrag,color}
\usepackage{enumitem,setspace}
\setenumerate{leftmargin=*}

\usepackage[colorlinks,bookmarks,linkcolor=black,citecolor=black]{hyperref}

\usepackage[notref,notcite,color]{showkeys}

\let\subset\subseteq 
\let\eps\varepsilon
\let\rho\varrho
\def\dcup{\dot\cup}

\def\itmit#1{\rm ({\it #1\,})}
\def\rom{\itmit{\roman{*}}}
\def\abc{\itmit{\alph{*}}}

\newtheorem{theorem}{Theorem}
\newtheorem{lemma}[theorem] {Lemma}

\newtheorem{problem}[theorem]{Problem}

\newtheorem{definition}[theorem] {Definition} 
\newtheorem{setup}[theorem] {Setup}
\theoremstyle{remark}

\newtheorem{type}{Type} 
\newtheorem{prepround}{Preparation} 
\numberwithin{AuxiliaryCl}{theorem}

\newtheorem{AuxiliaryClM}{Claim}
\newcounter{maintheoremcounter}
\numberwithin{AuxiliaryClM}{maintheoremcounter}
\newtheorem{AuxiliarySubClaimM}[AuxiliaryClM]{Claim}

\newcounter{lemma23counter}
\numberwithin{AuxiliaryClL23}{lemma23counter}

\newcommand{\By}[2]{\overset{\mbox{\tiny{#1}}}{#2}}
\newcommand{\ByRef}[2]{   \By{\eqref{#1}}{#2} }

\newcommand{\gBy}[1]{     \By{#1}{>} }

\newcommand{\geByRef}[1]{ \ByRef{#1}{\ge} }
\def\endofFact{\hfill\scalebox{.6}{$\Box$}}

\newcommand{\NATS}{\mathbb{N}}

\newcommand{\T}{\mathcal{T}}
\newcommand{\cS}{\mathcal{S}}
\newcommand{\cK}{\mathcal{K}}
\newcommand{\cD}{\mathcal{D}}
\newcommand{\cN}{\mathcal{N}}
\newcommand{\cZ}{\mathcal{Z}}
\newcommand{\M}{\mathcal{M}}
\newcommand{\I}{\mathcal{I}}

\def\A{A} 
\newcommand{\ttt}{\tau}
\newcommand{\mmm}{\mu}
\newcommand{\iii}{\iota}

\newcommand{\oldqed}{}
\newenvironment{factproof}[1][Proof]{
  \renewcommand{\oldqed}{\qedsymbol}
  \renewcommand{\qedsymbol}{\endofFact}
  \begin{proof}[#1]
}{
  \end{proof}
  \renewcommand{\qedsymbol}{\oldqed}
}

\newcommand{\conn}{\leadsto}
\newcommand{\EMAIL}[1]{  \textit{E-mail}: \texttt{#1}} 

\let\oldmarginpar\marginpar
\renewcommand\marginpar[1]{\-\oldmarginpar[\raggedleft\footnotesize #1]%
{\raggedright\footnotesize #1}}

\title{A density Corr\'adi--Hajnal Theorem}
  \author[P. Allen]{Peter Allen*}
  \author[J. B\"ottcher]{Julia B\"ottcher*}

  \thanks{
    *
    Department of Mathematics,
    London School of Economics,
    Houghton Street,
    London, WC2A 2AE, UK.
    \EMAIL{p.d.allen|j.boettcher@lse.ac.uk}}

  \author[J. Hladk\'y]{Jan Hladk\'y\dag}

  \thanks{
    \dag\ 
    DIMAP and Mathematics Institute,
    University of Warwick,
    Coventry, CV4~7AL, UK. The author is an EPSRC Research Fellow.
    \EMAIL{honzahladky@gmail.com}}

  \author[D. Piguet]{Diana Piguet\ddag}

\thanks{\ddag\ 
New Technologies for Information Society,
University of West Bohemia,
Pilsen, Czech Republic.
\EMAIL{piguet@ntis.zcu.cz}}
  \thanks{
    PA, JH, and DP were supported by DIMAP, EPSRC award
    EP/D063191/1.
    PA was partially supported by FAPESP (Proc.~2010/09555-7), and
    JB by FAPESP (Proc.~2009/17831-7).
    PA and JB are grateful to NUMEC/USP, N\'ucleo de Modelagem Estoc\'astica e
    Complexidade of the University of S\~ao Paulo, for supporting this
    research. The research leading to this result has received funding from the
    European Union's Seventh Framework Programme (FP7/2007-2013) under grant
    agreement no.\ PIEF-GA-2009-253925.
    An extended abstract of this paper appeared in the proceedings of the Eurocomb~2011 conference.
    Support by the Institut Mittag-Leffler (Djursholm, Sweden) is gratefully
acknowledged.}
\subjclass[2010]{05C35 (primary)}

\begin{document}

\begin{abstract}
    We find, for all sufficiently large $n$ and each $k$, the maximum number of edges in an $n$-vertex graph which does not contain $k+1$ vertex-disjoint triangles.
  
This extends a result of Moon [Canad. J. Math. 20 (1968), 96--102] which is in turn an extension of Mantel's Theorem. Our result can also be viewed as a density version of the Corr\'adi--Hajnal Theorem.
\end{abstract}

\maketitle


\section{Introduction}
A classic result of Mantel asserts that each $n$-vertex graph
$G$ with more than $\lfloor \frac n2\rfloor \lceil \frac n2 \rceil$ edges
contains a triangle. What can we say about the number of triangles in a graph with more than $\lfloor \frac n2\rfloor \lceil \frac n2 \rceil$ edges?

There are three natural interpretations of this question. We can ask how many \emph{vertex-disjoint} triangles are guaranteed, how many \emph{edge-disjoint} triangles are guaranteed, or simply how many triangles are guaranteed in \emph{total}. The answer to each of the first two questions is $1$ (which is trivial) and Rademacher proved (see~\cite{ErdosRademacher}) that the answer to the last is $\lfloor\tfrac{n}{2}\rfloor$; in each case the extremal example consists of a complete balanced bipartite graph with one edge added to the larger part. It is then natural to ask the same questions of $n$-vertex graphs $G$ with at least $\lfloor \frac n2\rfloor \lceil \frac n2 \rceil+m$ edges, for any $m\ge 1$.

These questions are much harder. Lov\'asz and Simonovits~\cite{LovSim} gave a conjectured lower bound on the number of triangles present in any $n$-vertex graph $G$ with at least $\lfloor \frac n2\rfloor \lceil \frac n2 \rceil+m$ edges, which Erd\H{o}s~\cite{ErdosRademacher} had already proved correct for $m$ small enough compared to $n$. The conjecture remains open, although a celebrated recent result of Razborov~\cite{RazborovTriangles}, using his method of flag algebras, is that the conjectured lower bound---a complicated continuous but only piecewise differentiable function in~$m$---is asymptotically correct for all $m$. The number of edge-disjoint triangles was studied by Gy\H{o}ri~\cite{Gyo:EdgeDisjoint}, but only for $m\le 2n-10$ were exact results proved, and for large $m$ it is not clear what the right answer should be.

In this paper we solve (for sufficiently large $n$) the problem of how many vertex-disjoint triangles are guaranteed to exist in an $n$-vertex graph $G$ with a given number of edges. It is convenient to rephrase the problem in the following way. 

\begin{problem}\label{prob:triangles}
  How many edges can an $n$-vertex graph $G$ possess if it does not contain $k+1$ vertex-disjoint triangles?
\end{problem}

This problem was first studied by Erd\H{o}s~\cite{Erdos62Tiling} and by Moon~\cite{Moon68Tiling}. The former
proved the exact result when $n\ge 400k^2$, and the latter when $n\ge 9k/2+4$, giving the following theorem.

\begin{theorem}[Moon~\cite{Moon68Tiling}]\label{thm:Moon}
  Suppose that $n \ge  9k/2 + 4$. Let $G$ be an $n$-vertex graph which does not contain $k+1$ vertex-disjoint triangles. Then 
  \[e(G)\le \binom{k}{2}+k(n-k)+\left\lceil\frac{n-k}{2}\right\rceil\left\lfloor\frac{n-k}{2}\right\rfloor \,.\]
\end{theorem}

Interestingly, although Moon states that his result ``almost certainly remains
valid for somewhat smaller values of $n$ also'',
in fact he almost reaches a natural barrier: the graph which Moon proved to
be extremal (the graph $E_1(n,k)$ in Definition~\ref{def:extremal}, see also Figure~\ref{fig:extremalgraphs}) is only
extremal when $n\ge 9k/2+3$.

We give an exact solution to Problem~\ref{prob:triangles} for all values of~$k$ when~$n$ is greater than an absolute constant~$n_0$.
Our main result, Theorem~\ref{thm:main}, states that the answer is given by four different extremal (families of) graphs in four different regimes of~$k$.

\medskip

We remark that our result can also be seen as a variation of two other classical theorems in extremal graph theory. Firstly, Erd\H{o}s and Gallai~\cite{ErdGall59} answered the analogous question for edges instead of triangles.

\begin{theorem}[Erd\H{o}s and Gallai~\cite{ErdGall59}]\label{thm:ErdGal}
  For any $n$-vertex graph~$G$ without $k+1$ vertex-disjoint edges,
  $e(G)\le \max\{k(n-k)+\binom{k}{2},\binom{2k+1}{2} \}$.
\end{theorem}

In fact, they showed that, depending on~$k$, the extremal graph for this problem either consists of $k$ vertices which are complete to all vertices, or of a $(2k+1)$-clique and a disjoint independent set. An analogous behaviour of the extremal structure in the hypergraph case is predicted by the famous Matching Conjecture of Erd\H{o}s,~\cite{Erdos:MatchingConjecture}. In this sense the appearance of various very different extremal structures in our result is not surprising.

Secondly, Corr\'adi and Hajnal~\cite{CorHaj} considered the variant of Problem~\ref{prob:triangles} where the number of edges is replaced by the minimum degree and proved the following well-known theorem.

\begin{theorem}[Corr\'adi and Hajnal~\cite{CorHaj}]\label{thm:corhaj}
  For any $n$-vertex graph~$G$ which does not contain $k+1$ vertex disjoint triangles, $\delta(G)\le k+\big\lfloor\tfrac{n-k}{2}\big\rfloor$.
\end{theorem}

The graph~$E_1(n,k)$ from Definition~\ref{def:extremal} is also extremal in this setting for the whole range $k\in [0,\frac{n}{3}]$.

Thus, our result is the `density version' of the Corr\'adi--Hajnal Theorem.

\subsection{Organisation of the paper} We state and discuss our main result,
Theorem~\ref{thm:main}, in Section~\ref{sec:result}. We outline its proof in Section~\ref{sec:outline}. The main combinatorial work of the proof is to be found in Sections~\ref{sec:small} and~\ref{sec:large}. In Section~\ref{sec:proof} we show how to deduce Theorem~\ref{thm:main} from these combinatorial arguments and some maximisation problems. In Section~\ref{sec:fewtriangles} we prove an auxiliary lemma which is one of the key points of the proof of Theorem~\ref{thm:main}, building on our previous work~\cite{AllBoettHlaPig_Strengthening}. In Section~\ref{sec:conclusion} we then discuss possibilities of extending our result. Our proof of Theorem~\ref{thm:main} requires tedious maximisation arguments, which we state as they are needed but whose derivations are postponed to Appendix~\ref{sec:max}. 

The proof relies on a number of elementary but lengthy calculations. These calculations were performed by hand, and the details are given. However, for verification and for the reader's convenience, we used the computer algebra software Maxima to check many of these calculations. The output pdf file as well as all the data in the wxMaxima format are available as ancillary files on the arXiv.

\section{Our result}\label{sec:result}
Given an integer $\ell$ and a graph $H$, we write $\ell\times H$ to denote the
disjoint union of $\ell$ copies of $H$. We say that a graph is \emph{$\ell\times
H$-free} if it does not contain~$\ell$ vertex disjoint (not necessarily induced)
copies of $H$. In Theorem~\ref{thm:main} we determine the maximal number of edges in a
$(k+1)\times K_3$-free graph on~$n$ vertices for every
$0\leq k<\frac{n}{3}$. The extremal formula is a somewhat opaque maximum of four different terms, so in preference to presenting it we shall describe four constructions of $n$-vertex $(k+1)\times K_3$-free graphs corresponding to these four terms. We say that an edge~$e$ (or more generally a set of
vertices) \emph{meets} a set of vertices~$X$ if~$e$ and~$X$ intersect. The
edge~$e$ meets~$X$ in~$X'$ if $X'=X\cap e$.

\begin{definition}[extremal graphs]
\label{def:extremal}
  Let $n$ and $k$ be non-negative integers with $k\le\frac{n}{3}$. We
  define the following four graphs (see also
  Figure~\ref{fig:extremalgraphs}).\footnote{The constructions for
    $E_2(n,k)$ and $E_4(n,k)$ do not give unique graphs. We collectively
    denote all graphs constructed in this way by $E_2(n,k)$ and $E_4(n,k)$,
    respectively. In the following we only use properties of these graphs that
    are shared by all of them.}
  \begin{enumerate}[label={$E_{\arabic{*}}(n,k)$:}]
  \item Let $X\dcup Y_1\dcup Y_2$ with $|X|=k$,
    $|Y_1|=\lceil\frac{n-k}{2}\rceil$, and
    $|Y_2|=\lfloor\frac{n-k}{2}\rfloor$ be the vertices of $E_1(n,k)$. Insert
    all edges intersecting
    $X$, and between~$Y_1$ and~$Y_2$.
  \item  The second class of extremal graphs is defined only for $k<\frac{n-1}4$.
    Let $X\dcup Y_1\dcup Y_2$ with $|X|=2k+1$,
    $|Y_1|=\lfloor\frac{n}{2}\rfloor$, and $|Y_2|=\lceil\frac{n}{2}\rceil-2k-1$ (or
    $|Y_1|=\lceil\frac{n}{2}\rceil$, and
    $|Y_2|=\lfloor\frac{n}{2}\rfloor-2k-1$) be the vertices of $E_2(n,k)$.
    Insert all edges within $X$, and between $Y_1$ and $X\cup Y_2$.  
    If~$n$ is odd, this construction captures two graphs, if~$n$ is even
    just one.

  \item Let $X\dcup Y_1$ with $|X|=2k+1$ and $|Y_1|=n-2k-1$  be the vertices of
    $E_3(n,k)$. Insert all edges intersecting $X$.

   \item
    The fourth class of extremal graphs is defined only for $k\ge \frac
    n6-2$. When $k\ge \frac{n-2}3$ take $E_4(n,k)$ to be the complete graph $K_n$. Otherwise, the vertex set is formed by five disjoint sets $X$, $Y_1$, $Y_2$,
    $Y_3$, and $Y_4$, with $|Y_1|=|Y_3|$, $|Y_2|=|Y_4|$,
    $|Y_1|+|Y_2|=n-3k-2$, and $|X|=6k-n+4$. Insert  all edges in
    $X$, between $X$ and $Y_1\cup Y_2$, and between $Y_1\cup Y_4$ and
    $Y_2\cup Y_3$. Thus the choice of $|Y_1|$ determines a
    particular graph in the class $E_4(n,k)$. All
    graphs in $E_4(n,k)$ have the same number of edges.
  \end{enumerate}
\end{definition}
\begin{figure}[ht]
\centering
\psfrag{E1}{\scalebox{1.1}{$E_1(n,k)$}}
\psfrag{E2}{\scalebox{1.1}{$E_2(n,k)$}}
\psfrag{E3}{\scalebox{1.1}{$E_3(n,k)$}}
\psfrag{E4}{\scalebox{1.1}{$E_4(n,k)$}}
\psfrag{X}{$X$}
\psfrag{Y1}{$Y_1$}
\psfrag{Y2}{$Y_2$}
\psfrag{Y3}{$Y_3$}
\psfrag{Y4}{$Y_4$}
\includegraphics[scale=0.4]{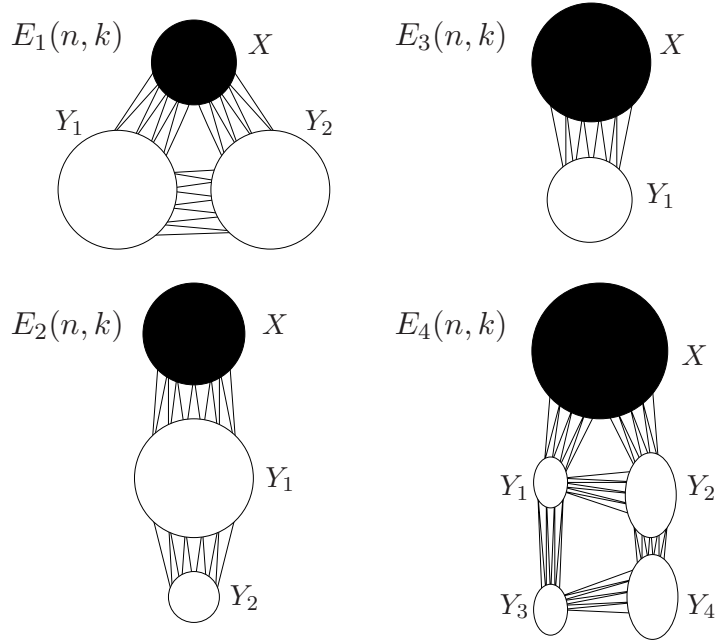}
\caption{The extremal graphs.}
\label{fig:extremalgraphs}
\end{figure}

Our main result is the following.

\begin{theorem}\label{thm:main} 
\setcounter{maintheoremcounter}{\value{theorem}}
  There exists $n_0$ such that for each $n>n_0$ and each $k$, $0\le k\le \frac n3$ we have the following. 
Let~$G$ be a
  $(k+1)\times K_3$-free graph on~$n$ vertices. Then
  \begin{equation}\label{eq:mainformula}
  e(G)\le\max_{j\in[4]}\,e\big( E_j(n,k) \big)\;.
  \end{equation}
\end{theorem}

For three sets $A,B,C$ (not necessarily distinct) we say that a triangle $uvw$ is of \emph{type} $ABC$, if $u\in A, v\in B$, and $w\in C$. 
All triangles in $E_1(n,k)$ are of type $XY_1Y_2$, $XXY_i, i=1,2$, or $XXX$, thus intersecting~$X$ at least once. 
Thus $E_1(n,k)$ contains at most $k$ vertex-disjoint triangles. All triangles in $E_2(n,k)$ and in $E_3(n,k)$ are of type $XXY_1$, or $XXX$, intersecting $X$ at least twice. Therefore these two graphs do not contain more than $k$ vertex-disjoint triangles. All triangles in $E_4(n,k)$ must be fully contained in $X\cup Y_1\cup Y_2$ and therefore there are at most $\lfloor\frac 13 |X\cup Y_1\cup Y_2|\rfloor=k$ vertex-disjoint triangles.

The graphs $E_i(n,k)$ are
edge-maximal subject to not containing $(k+1)\times K_3$. The only exception is $E_4(n,k)$ for $k\lesssim \frac n4$.  Indeed, when $k<\frac n4-1$, we have $|X|<2k$. Therefore, in any collection of $k$ vertex-disjoint triangles, there must be at least $w=2k-|X|$ triangles of type $XY_1Y_2$. Thus, if $|Y_1|\le w$, one can actually add edges inside $Y_1$ without increasing the maximum number of vertex-disjoint triangles.
However $E_4(n,k)$ is in any case not the extremal graph in this range; see the discussion below and Table~\ref{tab:transitions}. The graphs $E_i(n,k)$ have the following numbers of edges (after an exact formula we identify the leading terms; to this end we use the symbol~$\approx$).
\begin{align}
\begin{split}\label{eq:defEi}
e\big(E_1(n,k)\big) &=
\binom{k}{2}+k(n-k)+\left\lceil\frac{n-k}{2}\right\rceil\left\lfloor\frac{n-k}{2}\right\rfloor
\approx \tfrac14 n^2-\tfrac14 k^2+\tfrac12 kn \;,
\\
e\big(E_2(n,k)\big) &=
\binom{2k+1}{2}+\left\lceil\frac{n}{2}\right\rceil\left\lfloor\frac{n}{2}\right\rfloor
\approx \tfrac14 n^2+2k^2 \;,
\\ 
e\big(E_3(n,k)\big) &= 
\binom{2k+1}{2}+(2k+1)(n-2k-1) 
\approx 2kn-2k^2 \;,
\\
e\big(E_4(n,k)\big) &= 
\binom{6k-n+4}{2}+(6k-n+4)(n-3k-2)+(n-3k-2)^2 \\
& \approx \frac{n^2}2-3kn+9k^2 \;.
\end{split}
\end{align}
%
%
Comparing these edge numbers reveals that, as~$k$ grows from~$0$ to $n/3$,
the extremal graphs dominate in the following order (for $n$ sufficiently large).
In the beginning $E_1(n,k)$ has the most edges of these four graphs 
until $k\approx\frac{2n}{9}$, where it is surpassed
by $E_2(n,k)$. At $k\approx\frac{n}{4}$
this extremal structure ceases to exist and is replaced by $E_3(n,k)$,
until finally at $k\approx(5+\sqrt{3})n/22$ the graph
$E_4(n,k)$ takes over. The exact thresholds are listed in
Table~\ref{tab:transitions}. Further, the edge numbers of the
graphs $E_i(n,k)$ are plotted in
Figure~\ref{fig:plotExtremalGraphs}. \begin{figure}[ht]
\centering
\includegraphics[width=\textwidth]{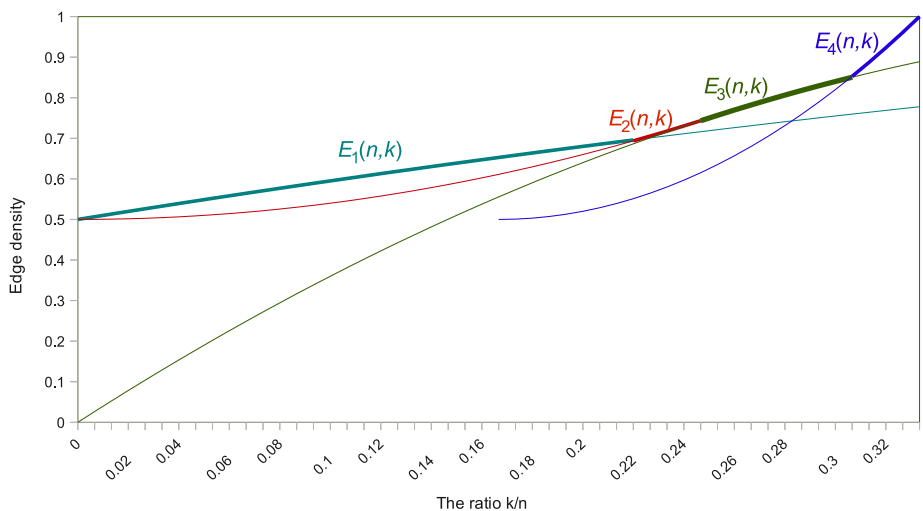}
\caption{
Edge densities of the graphs $E_i(n,k)$ where $k$ ranges from $0$ to $\frac n3$.
}
\label{fig:plotExtremalGraphs} 
\end{figure}

Observe that for fixed $n$, as $k$ increases, the transitions of the extremal graphs from
$E_1(n,k)$ to $E_2(n,k)$ and from $E_3(n,k)$ to $E_4(n,k)$
are not continuous:
$\Theta(n^2)$ edges must be edited to change from the former to the
latter structure. 
The transition from $E_2(n,k)$ to $E_3(n,k)$ however is continuous. 

{ \renewcommand\arraystretch{1.4}
\begin{table*}[ht]
   \begin{center}
    \begin{tabular}{|l|c|}
      \hline
      \bf graph & \bf extremal for \\ \hline
      $
      \displaystyle
      \begin{aligned}
        & E_1(n,k) \vphantom{\Bigg(} \\
        & E_2(n,k) \vphantom{\frac{2n-6}{9}} \\
        & E_3(n,k) \vphantom{\frac{\sqrt{3n^2}}{22}} \\
        & E_4(n,k) \vphantom{\Bigg(}
      \end{aligned}
      $
      &
      $
      \displaystyle
      \begin{gathered}
        1 \le  k\le \frac{2n-6}{9} \vphantom{\Bigg(} \\
        \frac{2n-6}{9} \le  k\le\frac{n-1}{4} \\
        \frac{n-1}{4} \le  k\le \frac{5n-12+\sqrt{3n^2-10n+12}}{22}\quad \\
        \quad\frac{5n-12+\sqrt{3n^2-10n+12}}{22} \le k \le \frac{n}{3} \vphantom{\Bigg(}
      \end{gathered}
      $
      \\ \hline
    \end{tabular}
  \end{center}
\caption{Transitions between the extremal graphs.}
\label{tab:transitions}
\end{table*}
}

\section{Proof outline and setup}\label{sec:outline}

The basic idea of our proof is straightforward: we show that we can partition the
vertices of any $(k+1)\times K_3$-free graph into six parts, and establish some
upper bounds on the numbers of edges within and between these parts in terms of
their sizes only. This defines a function (of six variables) which is an upper
bound on the number of edges of a graph with parts of the given sizes. Then
maximising this function (subject to $n$ and $k$ being fixed) we obtain an upper
bound on the number of edges of a $(k+1)\times K_3$-free graph with $n$ vertices,
and observe that this matches the lower bounds provided by the extremal
structures given in Definition~\ref{def:extremal}.

\smallskip

We shall now fix the basic setup for our proof, i.e., we will specify the
above mentioned six parts, which will be called $\T_1$, $\T_2$, $\T_3$,
$\T_4$, $\M$, and $\I$. We need the following definition. Let~$G$ be a
graph, $uv$ be an edge in~$G$ and $xyz$ a triangle in~$G$. We say that $uv$
\emph{sees} vertex $x$ of $xyz$ if $uvx$ is a triangle
in $G$. The edge $uv$ sees $xyz$ if $uv$ sees at least one of the vertices $x$, $y$, or~$z$.
Similarly, we say that a vertex $u$ sees (the edge $xy$ of) the
triangle $xyz$ if $uxy$ is a triangle in~$G$.

Throughout we will assume the following setup.

\begin{setup}
\label{setup}\index{$\T$}\index{$\M$}\index{$\I$}\index{$\T_1,\ldots,\T_4$}
  Let~$G$ be an $n$-vertex graph which is edge-maximal subject to not
  containing $(k+1)\times K_3$. Let $\T$ be a set of $k$ vertex-disjoint
  triangles in~$G$, let~$\M$ be a maximum matching outside $\T$, and
  presume $\T$ is chosen to maximise the size of $\M$. The
  remaining vertices of $G$, which form an independent set, we call~$\I$.

  We now split the set $\T$ into four parts as follows, forming together with~$\M$ and $\I$ the six above-mentioned parts of $G$.
  Let $\T_1$ be the set of triangles in~$\T$ seen by at least two $\M$-edges.
  Let $\T_2$ be the set of triangles in $\T-\T_1$ seen by either an
  $\M$-edge and at least one $\I$-vertex or by two $\I$-vertices.
  Finally, we aim to partition the remaining triangles of $\T$ into a `sparse part'
  $\T_3$ and a `dense part' $\T_4$ by applying the following algorithm.
  We start with $D$ equal to the set of all triangles in $\T-(\T_1\cup\T_2)$,
  and $S=\emptyset$. If there is a triangle in $D$ which sends at most
  $8(|D|-1)$ edges to the other triangles in $D$, we move it to $S$. (Consequently each triangle in~$D$ sends at least~$8$ edges to other single triangles in~$D$ on average.) We
  repeat until $D$ contains no more such triangles. We then set $\T_3:=S$,
  and $\T_4:=D$. Note that every triangle in $\T_4$ sends more than $8(|\T_4|-1)$
  edges to the other triangles in $\T_4$.

  \index{$m$}\index{$i$}\index{$t_1\ldots,t_4$}
  We define $m:=|\M|$, $i:=|\I|$, and $t_j:=|\T_j|$ for all
  $j\in[4]$.
\end{setup}

We remark that the outcome of the algorithm for constructing $\T_3$ and
$\T_4$ is not uniquely determined. However, any possible pair $\T_3$ and
$\T_4$ resulting from the construction we described is suitable for our
purposes.

Further, we emphasise that $k=|\T|$ is the number of triangles in $\T$,
which cover $3k$ vertices (and similarly $\M$ covers $2m$ vertices). The
function $e(\bullet)$ counts the number of edges in $G$ induced by the
structure~$\bullet$, e.g., $e(\T_3)=e(G[V(\T_3)])$. Similarly,
$e(\bullet,\star)$ counts edges in the bipartite graph between the
structures~$\bullet$ and~$\star$.

Before we proceed, let us give some motivation for the above defined
partition of~$G$ by applying it to our four extremal graphs from
Definition~\ref{def:extremal}. First consider the graph $E_1(n,k)$. It is
easy to check that for this graph we have $\T=\T_1$  in the range when this graph is optimal (see Table~\ref{fig:extremalgraphs}), and all vertices
(except perhaps one) outside~$\T$ are in~$\M$. Any pair of
triangles of $\T$ has seven edges between them in $E_1(n,k)$, the set $\M$ induces $m^2$ edges, and $e(\M,\T_1)=4mt_1$. We shall show in our proof that \emph{in any graph~$G$}, the definition of $\T_1$ forces that
any two triangles of~$\T_1$ have at most seven edges between them (see Lemma~\ref{clm:NRupper}\ref{T1}), the set~$\M$ induces at most $m^2$ edges (see Lemma~\ref{clm:NRupper}\ref{M}), and $e(\M,\T_1)\le 4mt_1$ (see Lemma~\ref{clm:NRupper}\ref{MT1}). Together with bounds which we will prove on the number of edges touching $\I$, we conclude that if $\T=\T_1$ then $e(G)\le e\big(E_1(n,k)\big)$.

Similarly, the definition of $\T_2$ and $\T_4$ is motivated by the fact
that in both $E_2(n,k)$ and $E_3(n,k)$ we have $\T=\T_2$, while in
$E_4(n,k)$ we have $\T=\T_4$ (again in the appropriate range). The set $\T_3$ is always empty in the
extremal graphs. It turns out that, for $E_2(n,k)$ and $E_3(n,k)$ we will
be able to use a similar strategy as lined out for $E_1(n,k)$, i.e., we
shall infer from the definition of~$\T_2$ that $E_2(n,k)$ and $E_3(n,k)$
have a maximal number of edges in $\T_2$ (see Lemma~\ref{clm:NRupper}\ref{T2})
and then show that $\T=\T_2$ in an extremal graph (for the appropriate
range of $k$). For $E_4(n,k)$ we must work harder: the definition of~$\T_4$
permits nine edges to exist between a pair of triangles, yet in $E_4(n,k)$
only some pairs of triangles actually have nine edges between them (see Table~\ref{tab:NumberEdges}). 
{ \renewcommand\arraystretch{1.4}
\begin{table*}[ht]
   \begin{center}
    \begin{tabular}{c|c|c|c|c}
      & $XXX$ & $XXY_1$ & $XXY_2$ & $XY_1Y_2$       
\\      \hline
$XXX$ & 9 & 9 & 9 & 9\\ \hline
$XXY_1$& 9 & 8 & 9 & 8\\ \hline
$XXY_2$& 9 & 9 & 8 & 8\\ \hline
$XY_1Y_2$& 9 & 8 & 8 & 7 
    \end{tabular}
  \end{center}
\caption{Number of edges between triangles of different types in~$E_4(n,k)$.}
\label{tab:NumberEdges}
\end{table*}
}

\smallskip

As explained, our main goal in the following will be to establish bounds on
the number of edges within and between the six parts of~$G$. One concept
that will turn out to be very fruitful in this context is that of a rotation.

\begin{definition}[rotation]
  Let~$G'$ be a graph and let $\T'$ be a triangle factor in~$G'$.  An
  \emph{improving rotation} on a set $V'$ is a set of vertex
  disjoint triangles~$\tilde \T$ in~$V'$ which witnesses either that $\T'$ is not of maximum size, or that its choice does not maximise the
  matching number of $G'-V(\T')$: We can replace those triangles of~$\T'$ which are contained in~$V'$ by the triangles~$\tilde \T$
  and obtain a triangle factor~$\T''$ with one of the following two
  properties. Either $|\T''|>|\T'|$, or $|\T''|=|\T'|$ but the matching
  number of $G'-V(\T'')$ is bigger than that of $G'-V(\T')$.  If, on the
  other hand, $|\T''|=|\T'|$ and the matching number of $G'-V(\T'')$ equals
  that of $G'-V(\T')$ then~$V'$ is a \emph{non-improving rotation} or
  simply rotation.  In both cases we also say that we can \emph{rotate}
  from~$\T'$ to~$\T''$.
\end{definition}

Typically, the rotations that we will consider are local structures.  To
give an example, let~$G$ and~$\T$ be as in Setup~\ref{setup}. By
definition, there are no improving rotations in~$G$. Suppose, however, that
we find outside $\T$ two vertex-disjoint edges $uv$ and $u'v'$, and a
triangle $xyz$ of $\T$ with the property that $x$ is a common neighbour of
$uv$, and $y$ of $u'v'$. This structure allows us to rotate by replacing
$xyz$ with $uvx$ and $u'v'y$, a contradiction. The non-existence of this structure leads to an upper bound on the number of edges between $\M$ and $\T$.

\section{Small rotations}\label{sec:small}

In this section we will describe several rotations involving small numbers (one or two) of triangles, and show that their non-existence gives good bounds on the maximum number of edges within and between $\T_1$, $\T_2$, $\T_3$, $\M$ and $\I$. The bounds obtained on edges involving $\T_4$ are not strong enough for the proof of Theorem~\ref{thm:main}, but they are strong enough to prove the following lemma, which serves both as an illustration of our technique and as a necessary step in the proof of Theorem~\ref{thm:main}.

\begin{lemma}\label{lem:tri1} 
  Let $k\leq\frac{n-8}{5}$ be an integer and let $G$ be a
  $(k+1)\times K_3$-free graph on $n$ vertices. Then
  $e(G)\leq e\big(E_1(n,k)\big)$.
\end{lemma}

Observe that, in contrast to Theorem~\ref{thm:main} we do not require any
lower bound on~$n$ in this lemma. Observe also that since $\tfrac{n-8}{5}<\tfrac{2n-8}{9}$, the result follows from Theorem~\ref{thm:Moon}: but its proof will exemplify our techniques and put us into position to
explain the remaining steps to obtain Theorem~\ref{thm:main}.

We assume in the following Setup~\ref{setup}.
We start with some simple upper bounds.

\begin{lemma}\label{clm:NRupper} The following bounds hold.
\begin{enumerate}[label=\abc]
  \item $e(\I)=0$.\label{I}
  \item $e(\I,\M)\leq im$.\label{IM}
  \item $e(\M)\leq m^2$.\label{M}
  \item $e(\M,\T_1)\leq 4mt_1$.\label{MT1}
  \item $e(\I,\T_1)\leq 2it_1$.\label{IT1}
  \item $e(\T_1)\leq 7\binom{t_1}{2}+3t_1$.\label{T1}
  \item $e(\I,\T_2)\leq 2it_2$.\label{IT2}
  \item $e(\T_2)\leq 8\binom{t_2}{2}+3t_2$.\label{T2}
  \item $e(\T_3)+e(\T_3,\T_4)\leq
8\binom{t_3}{2}+8t_3t_4+3t_3$.\label{T3T4}
\end{enumerate}
\end{lemma}

\begin{proof}
We leave to the reader the proof of~\ref{I}.

Suppose that a vertex $u\in \I$ sends more than $m$ edges to $\M$. Then
there is some edge $vw$ of $\M$ which receives two edges from $u$. So $uvw$
is a triangle of $G$, contradicting maximality of $|\T|$. Summing over
vertices of $\I$, bound~\ref{IM} follows.
Similarly, if a vertex of $\M$ was adjacent to more than~$m$ other vertices
of $\M$ this would contradict maximality of $\T$. Bound~\ref{M} follows.

If an edge $uv$ of $\M$ sends more than four edges to any triangle $T$ of
$\T_1$, then it must see two vertices of $T$. Since by definition of
$\T_1$ there is another edge $u'v'$ of $\M$ which sees a vertex of $T$,
there are two vertices $x,x'$ of $T$ such that $uvx$ and $u'v'x'$ are
triangles of $G$. This is an improving rotation which contradicts the maximality of
$\T$. Therefore, no such edge exists. Bound~\ref{MT1} follows by summation.
Similarly, if a vertex $u$ of $\I$ were to send three edges to a triangle $T$
of $\T_1$, then (using an edge of $\M$ which sees $T$) we would have an improving rotation
increasing the size of $\T$. Bound~\ref{IT1} follows.

Now suppose there were two triangles $uvw$ and $u'v'w'$ of $\T_1$ with more than
seven edges between them. By definition of $\T_1$ we can find disjoint edges
$xy$ and $x'y'$ of $\M$ such that $xy$ sees $u$ and $x'y'$ sees $u'$. Because
there are at least eight edges between $uvw$ and $u'v'w'$, there must be at
least three edges between $vw$ and $v'w'$. In particular, there is a triangle
contained in $\{v,w,v',w'\}$. Together with $xyu$ and $x'y'u'$ this is an improving
rotation increasing $\T$, contradicting the maximality of $|\T|$. This implies
bound~\ref{T1}.

Next, suppose there is a vertex $u$ of $\I$ which sends three edges to a
triangle $xyz$ of $\T_2$. We utilise the definition of $\T_2$ and infer that one of the two cases must occur. Either there is a second vertex $u'$
of $\I$ which sees two vertices $\{x,y\}$ of that triangle. 
Hence we can rotate and replace the triangle $xyz$ and the vertices $u$ and $u'$
by the triangle $xyu'$ and the edge $uz$, a contradiction. The other case when $xyz$ is seen by an edge of $\M$ can be treated similarly.
It follows that no vertex of
$\I$ sends three edges to any triangle of $\T_2$, hence bound~\ref{IT2}.

We now turn to proving~\ref{T2}. Suppose that there is a pair of triangles $xyz$ and $x'y'z'$ of $\T_2$ forming a copy of $K_6$. By the definition of $\T_2$ we either have that there are distinct vertices $u,u'\in\I$ which see respectively $xy$ and $x'y'$, or that there is a vertex $u\in I$ which sees $xy$ and an edge $ab\in \M$ disjoint from $u$ which is seen by $x'$. Suppose the former case. Then we have a similar improving rotation as above: we form $xyu$, $x'y'u'$, and $zz'$, a contradiction. An analogous improving rotation exists in the other case. This yields our bound~\ref{T2}.

Finally, we must show that $e(\T_3)+e(\T_3,\T_4)\leq
8\binom{t_3}{2}+8t_3t_4+3t_3$. This bound does not come from a
rotation. Instead, recall that $\T_3$ is formed sequentially. We claim that the
bound applies to every pair of sets $S$ and $D$ during the construction in Setup~\ref{setup}, that is,
that $e(S)+e(S,D)\le8\binom{|S|}{2}+8|S||D|+3|S|$. This is
trivially true at the first stage, when $S=\emptyset$. Now a triangle is moved from
$D$ to $S$ when it sends at most $8(|D|-1|)$ edges to the rest of $D$. So $|S|$
is increased by one, and $e(S)+e(S,D)$ is increased by at most $3+8(|D|-1)$.
Bound~\ref{T3T4} follows by induction.
\end{proof}

We next come to two bounds on edges within $\T$.

\begin{lemma}\label{clm:TiTj} The following bounds hold.
\begin{enumerate}[label=\abc,start=10]
  \item When $t_1\neq 1$ and $j\ge 2$, then $e(\T_1,\T_j)\leq 7t_1t_j$.\label{T1Ti}
  \item When $t_2\neq 1$ and $j\ge 3$, then $e(\T_2,\T_j)\leq 8t_2t_j$.\label{T2Ti}
\end{enumerate}
\end{lemma}

\begin{proof}
  We first show~\ref{T1Ti}. Since the case $t_1=0$ is trivial, we assume that
  $t_1\ge 2$. Let $xyz$ be a triangle in $\T_j$, for some $j\ge 2$, and suppose
  that there are at least $7t_1+1$ edges from $\T_1$ to $xyz$. Then certainly
  there is a triangle $uvw\in\T_1$ which sends at least eight edges to $xyz$.
  There are two possibilities.
  
  First, suppose $uvw$ sends exactly eight edges to
  $xyz$. Then there is another triangle $u'v'w'\in\T_1$ which sends at least
  seven edges to $xyz$. By definition of $\T_1$, there are distinct edges $ab$
  and $a'b'$ of $\M$ such that $ab$ sees $u$ and $a'b'$ sees $u'$. Since there
  are seven edges from $u'v'w'$ to $xyz$, $v'w'$ must have a common neighbour
  $x$; since there are eight edges from $xyz$ to $uvw$, $yz$ must have two
  common neighbours in $uvw$, and in particular one, say $v$, which is not $u$.
  Then replacing $uvw$, $u'v'w'$ and $xyz$ with $abu$, $a'b'u'$, $v'w'x$ and
  $yzv$ is an improving rotation, a contradiction.
  
  Second, suppose $uvw$ sends nine edge to $xyz$. Then there is another triangle
  $u'v'w'$ of $\T_1$ which sends at least six edges to $xyz$. Again we assume
  $ab\in\M$ sees $u$, and $a'b'\in\M$ sees $u'$. Now at least one of $v'$ and
  $w'$, say~$v'$, must have two neighbours in $xyz$, say $x$ and $y$. Since
  $xyz$ sends nine edges to $uvw$, $zvw$ is a triangle. Then replacing $uvw$,
  $u'v'w'$ and $xyz$ with $abu$, $a'b'u'$, $v'xy$ and $zvw$ is an improving
  rotation, a contradiction. The bound~\ref{T1Ti} follows by summation.
  
  We now show~\ref{T2Ti}. Again, we assume $t_2\ge 2$ and suppose $xyz\in\T_j$
  for some $j\ge 3$ sends at least $8t_2+1$ edges to $\T_2$. Then there are
  triangles $uvw$ and $u'v'w'$ of $\T_2$ to which $xyz$ sends respectively nine
  and at least eight edges. 
  We now use the fact that $uvw,u'v'w'\in\T_2$ to infer the following: either there are distinct vertices
$a$ and $a'$ of $\I$ which see respectively~$uv$ and~$u'v'$, or there is a vertex $a\in\I$ and an edge $bc\in\M$ such that $a$ sees $uv$ and~$u'$ sees $bc$. Let us consider the first case. Now $w'$ is
  adjacent to at least two vertices of $xyz$, say $x$ and $y$, and $zw$ is an
  edge. Therefore replacing $uvw$, $u'v'w'$ and $xyz$ by $auv$, $a'u'v'$ and
  $w'xy$ maintains the number of triangles of $\T$, but allows us to add $zw$ to
  $\M$, and is thus an improving rotation, a contradiction. Next we consider the case when there is a vertex $a\in\I$ and an edge $bc\in\M$ such that $a$ sees $uv$ and $u'$ sees $bc$. There is a vertex of $xyz$ which sees $v'w'$, say $x$. Then replacing $uvw$, $u'v'w'$ and $xyz$ by $bcu'$, $v'w'x$, $yzw$, and $auv$ is an improving rotation, again a contradiction. The bound~\ref{T2Ti} follows by
  summation.
\end{proof}

Our next task is to bound the edges between $\M$ and $\T_j$, $j\ge 2$, and
between $\I$ and $\T_j$, $j\ge 3$. We combine these bounds with those given in
Lemma~\ref{clm:TiTj} because they permit us to handle the cases $t_1=1$ and
$t_2=1$ which were not dealt with in Lemma~\ref{clm:TiTj}. However, in the proof
of Theorem~\ref{thm:main} we will find that we require both sets of bounds.

\begin{lemma}\label{clm:InTupper} 
The following bounds hold. 
\begin{enumerate}[label=\abc,start=12]
  \item \[e(\T_1,\T_2)+e(\M,\T_2)\leq
  \begin{cases}
    7t_1t_2+\left(2+3m\right)t_2 & \mbox{if } m\ge 1\; \mbox{and}\\
    0 & \mbox{if } m=0 \;.
  \end{cases}\]\label{T2,T1+M}
  \item When $j=3,4$ we have \[e(\T_1,\T_j)+e(\M,\T_j)\leq
  \begin{cases}
    7t_1t_j+\left(3+3m\right)t_j & \mbox{if } m\ge 1 \; \mbox{and}\\
    0 & \mbox{if } m=0\;.
  \end{cases}\]\label{Ti,T1+M}
  \item When $j=3,4$ we have \[e(\T_2,\T_j)+e(\I,\T_j)\leq
  \begin{cases}
    8t_2t_j+\left(2+i\right)t_j & \mbox{if } i\ge 1\; \mbox{and}\\
    0 & \mbox{if } i=0\;.
  \end{cases}\]\label{Ti,T1+I}
\end{enumerate}
\end{lemma}

\begin{proof}
  First we prove~\ref{T2,T1+M}. Observe that if $m=0$ then by
  definition of $\T_1$ we have also $t_1=0$, and the bound follows. Now by
  definition of $\T_2$, any triangle $xyz\in\T_2$ is seen by at most one edge $ab$
  in $\M$. It follows that all other edges of $\M$ send at most three edges to
  $xyz$. Furthermore, if $ab$ sent six edges to $xyz$, then we would find an
  improving rotation as follows. Let $c\in\I$ be a vertex which sees (say) the edge $xy$ in
  $xyz$, whose existence is guaranteed by definition of $\T_2$. Now $cxy$ and
  $abz$ are disjoint triangles which can replace $xyz$ to increase the size of $\T$. It
  follows that $xyz$ sends at most $5+3(m-1)=3m+2$ edges to $\M$.
  
  If $t_1\neq 1$, then summing over $\T_2$ together with the bound~\ref{T1Ti} of Lemma~\ref{clm:TiTj}
  gives the desired bound~\ref{T2,T1+M}. If $t_1=1$, then we must work a little
  harder. Either $xyz\in\T_2$ sends at most seven edges to the triangle $uvw\in\T_1$,
  in which case $xyz$ sends in total at most $3m+7+2$ edges to $\T_1\cup\M$, or
  $xyz$ sends more than seven edges to $uvw$. In this case, we claim that no
  edge of~$\M$ sees~$xyz$, or we would have an improving rotation exactly as in
  the proof of bound~\ref{T1} of Lemma~\ref{clm:TiTj}. It follows that $xyz$ sends at most $3m$ edges to~$\M$, and so in total again at most $3m+9$ edges to $\T_1\cup\M$. Now
  summation yields the desired bound~\ref{T2,T1+M}.
  
  We next prove~\ref{Ti,T1+M}. Suppose $j\in\{3,4\}$. Again the $m=0$ case is
  trivial. Again by definition of $\T_j$, at most one edge in $\M$ sees the triangle $xyz\in\T_j$,
  and thus we have that $xyz$ sends at most $3m+3$ edges to $\M$. Again, if
  $t_1\neq 1$ then summation combined with the bound~\ref{T1Ti} of Lemma~\ref{clm:TiTj} yields the
  desired bound~\ref{Ti,T1+M}. Again, if $t_1=1$ then we either have that $xyz$
  sends at most seven edges to the triangle $uvw\in\T_1$, and so in total
  $3m+10$ edges to $\T_1\cup\M$, or it sends more than seven edges to $uvw$ but
  is not seen by any edge of $\M$ (or this would create an improving rotation), and so sends at most $3m+9$ edges to
  $\T_1\cup\M$. Again the desired bound~\ref{Ti,T1+M} follows by summation.
  
  Finally we prove the bound~\ref{Ti,T1+I}. Suppose $j\in\{3,4\}$. Observe that
  if $i=0$ then we have by definition of $\T_2$ that $t_2=0$ and hence the bound
  follows. Now by definition of $\T_j$, at most one vertex of $\I$ sees the
  triangle $xyz\in\T_j$, and all other vertices of $\I$ therefore send at most
  one edge to $xyz$. We conclude that $xyz$ sends at most $3+(i-1)=i+2$ edges
  to $\I$. If $t_2\neq 1$, then summation and the bound~\ref{T2Ti} of Lemma~\ref{clm:TiTj} yield the
  desired bound~\ref{Ti,T1+I}. If $t_2=1$, then there are two
  possibilities. First, $xyz$ sends at most eight edges to the triangle
  $abc\in\T_2$, in which case it sends in total at most $10+i$ edges to
  $\T_2\cup\I$. Second, $xyz$ sends nine edges to $abc$, in which case there can
  exist no vertex of $\I$ which sees $xyz$ or we would have an improving
  rotation exactly as in the proof of bound~\ref{T2} of Lemma~\ref{clm:TiTj}. Then $xyz$ sends in total
  at most $i+9$ edges to $\T_2\cup\I$. The desired bound~\ref{Ti,T1+I} follows
  by summation.
\end{proof}

Observe that, at this stage, we provided bounds for all (bipartite or internal)
edge sets but $e(\T_4)$. These bounds, with the exception of the bounds on edges
in $\T_4\cup\M\cup\I$, will turn out to be strong enough for all parts of the
proof of Theorem~\ref{thm:main}. It is convenient to summarise them in one
function. First, let
\begin{align}\label{eq:deffprime}\begin{split}
\index{$f'(\ttt_1,\ttt_2,\ttt_3,\ttt_4,\mmm,\iii)$}
    f'(\ttt_1,\ttt_2,\ttt_3,\ttt_4,\mmm,\iii):= &
    4\mmm\ttt_1+2\iii\ttt_1+7\binom{\ttt_1}{2}+3\ttt_1+2\iii\ttt_2\\ &
    + 8\binom{\ttt_2}{2}+3\ttt_2 +8\binom{\ttt_3}{2}+ 8\ttt_3\ttt_4+3\ttt_3\\
    &
    + 7\ttt_1\ttt_2+(2+3\mmm)\ttt_2
    +7\ttt_1(\ttt_3+\ttt_4) \\
    & +(3+3\mmm)\ttt_3+8\ttt_2(\ttt_3+\ttt_4)+(2+\iii)\ttt_3   \;.
\end{split}
\end{align}

We now define $f(\ttt_1,\ttt_2,\ttt_3,\ttt_4,\mmm,\iii)$ by
\begin{equation}\label{eq:deffixed}
  \begin{split}
  \index{$f(\ttt_1,\ttt_2,\ttt_3,\ttt_4,\mmm,\iii)$}
    f:=
    \begin{cases}
    f'(\ttt_1,\ttt_2,\ttt_3,\ttt_4,\mmm,\iii) & \mbox{when } \mmm\ge 1 \mbox{
    and } \iii\ge 1\\
    f'(\ttt_1,\ttt_2,\ttt_3,\ttt_4,\mmm,\iii)-(2\ttt_2+3\ttt_3) & \mbox{when } \mmm=0 \mbox{ and } \iii\ge 1\\
    f'(\ttt_1,\ttt_2,\ttt_3,\ttt_4,\mmm,\iii)-2\ttt_3 & \mbox{when } \mmm\ge 1
    \mbox{ and } \iii=0 \\
    f'(\ttt_1,\ttt_2,\ttt_3,\ttt_4,\mmm,\iii)-(2\ttt_2+5\ttt_3) & \mbox{when }
    \mmm=0 \mbox{ and } \iii=0 \\
  \end{cases}
  \end{split}
\end{equation}

The purpose of the functions $f$ and $f'$ is the following. When $t_1,t_2\neq 1$, we have by
summing the bounds in parts~\ref{MT1}--\ref{T3T4} of Lemma~\ref{clm:NRupper},
the $j=4$ cases of parts~\ref{T1Ti} and~\ref{T2Ti} of Lemma~\ref{clm:TiTj},
part~\ref{T2,T1+M} of Lemma~\ref{clm:InTupper} and the $j=3$ cases of parts~\ref{Ti,T1+M} and~\ref{Ti,T1+I}
of Lemma~\ref{clm:InTupper} that $$e(G)-e(\T_4\cup\M\cup\I)\le
f(t_1,t_2,t_3,t_4,m,i)\;.$$ We observe that the reason that $f$ and $f'$
differ is that Lemma~\ref{clm:InTupper} yields different bounds depending on
whether $m$ or $i$ is zero, i.e., we have 
$$e(G)-e(\T_4\cup\M\cup\I)\le
f'(t_1,t_2,t_3,t_4,m,i)\;.$$ We further observe that although $e(G)-e(\T_4\cup\M\cup\I)\le
f(t_1,t_2,t_3,t_4,m,i)$ is valid in general only when $t_1,t_2\neq 1$, by
parts~\ref{Ti,T1+M} and~\ref{Ti,T1+I} of Lemma~\ref{clm:InTupper} the following is always valid.
\begin{align}\label{eq:lpd}
\begin{split}
e(G)&-e(\T_4\cup\M\cup\I)+e(\T_4,\M\cup \I)\\
    & \le e(G)-e(\T_4\cup\M\cup\I)+e(\T_4,\M\cup \I)+e(\T_1\cup \T_2,\T_4) \\
    & \le f(t_1,t_2,t_3,t_4,m,i)+(3+3m)t_4+(2+i)t_4   \;.
\end{split}
\end{align}

As previously mentioned, our proof has a combinatorial part and an arithmetic
part: we need to know the maxima of several functions, of which $f$ is the
first. We state the required lemma here, but defer the proof to Appendix~\ref{sec:max}.
Let
\begin{align}\label{eq:DefF}\index{$F(n,k)$}
\begin{split}
  F(n,k):=\big\{&(\ttt_1,\ttt_2,\ttt_3,\ttt_4,\mmm,\iii)\in\NATS^6_0\,\colon\,\\
  &~~~\ttt_1+\ttt_2+\ttt_3+\ttt_4=k\,,2\mmm+\iii=n-3k\big\}\,.
\end{split}  
\end{align}

\begin{lemma}\label{lem:maxf}
  When $n\ge 3k+2$ we have
 \[\max_{(\ttt_1,\ttt_2,0,0,\mmm,\iii)\in
      F(n,k)}\big(f(\ttt_1,\ttt_2,0,0,\mmm,\iii)+\iii\mmm+\mmm^2\big)=\max_{j\in[3]}e\big(E_j(n,k)\big)\,.\]
\end{lemma}

A trivial upper bound for $e(\T_4)$ is given by
\begin{equation}\label{T4}
e(\T_4)\leq\binom{3t_4}{2}.
\end{equation}
It turns out that this trivial bound suffices to prove
Lemma~\ref{lem:tri1} (but not Theorem~\ref{thm:main}). We define $h(\ttt_1,\ttt_2,\ttt_3,\ttt_4,\mmm,\iii)$ by
\begin{equation}\label{eq:defh}
\index{$h(\ttt_1,\ttt_2,\ttt_3,\ttt_4,\mmm,\iii)$}
h:=f(\ttt_1,\ttt_2,\ttt_3,\ttt_4,\mmm,\iii)+\iii\mmm+\mmm^2+(3+3\mmm)\ttt_4+(2+\iii)\ttt_4+\binom{3\ttt_4}{2}\;.
\end{equation}


\begin{proof}[Proof of Lemma~\ref{lem:tri1}]
  Let $k\le\frac{n-8}5$. Let~$G$ and its decomposition be as in
  Setup~\ref{setup}. In particular, we obtain numbers $t_1,\ldots,t_4,m,i$.
  By~\eqref{eq:lpd}, \ref{I}--\ref{M} of Lemma~\ref{clm:NRupper}, and~\eqref{T4} 
  we have $e(G)\le h(t_1,t_2,t_3,t_4,m,i)$ for the function~$h$ defined in~\eqref{eq:defh}.
From~\eqref{eq:deffprime},~\eqref{eq:deffixed}, and~\eqref{eq:defh} one can check that
  \[h(t_1,t_2,0,t_3+t_4,m,i)\ge h(t_1,t_2,t_3,t_4,m,i)\;.\] 
  Also from~\eqref{eq:defh} we have the following.
  \begin{align}
  \nonumber
    &
    h(t_1+t_3+t_4,t_2,0,0,m,i)-h(t_1,t_2,0,t_3+t_4,m,i) \\
  \label{eq:KostaCafe}  
    & = (t_3+t_4)(m+i-t_2-t_3-t_4-4)\\
    \nonumber
    & \ge (t_3+t_4)\frac{n-5k-8}{2}\,,
  \end{align}
  where the inequality comes from $t_2+t_3+t_4\le k$ and $2m+i=n-3k$.
  Since $n-5k-8\ge 0$, we have
  \[h(t_1+t_3+t_4,t_2,0,0,m,i)\ge h(t_1,t_2,t_3,t_4,m,i)\,.\]
  Now $h(t_1+t_3+t_4,t_2,0,0,m,i)=f(t_1+t_3+t_4,t_2,0,0,m,i)+im+m^2$, so by
  Lemma~\ref{lem:maxf} we have \[e(G)\le
  h(t_1,t_2,t_3,t_4,m,i)\le\max_{j\in[3]}e\big(E_j(n,k)\big)\,.\]
  Finally, according to Table~\ref{tab:transitions}, this maximum is given by $e\big(E_1(n,k)\big)$, completing the proof.
\end{proof}

\section{Large rotations}
\label{sec:large}

In order to prove Theorem~\ref{thm:main} we need to improve the bounds given in the previous section on the number of edges touching $\T_4$; in particular, we need stronger bounds than the trivial $e(\T_4)\le\binom{3|\T_4|}{2}$. We will obtain these stronger bounds by describing rotations using many more---up to $29$---triangles. In constructing these rotations, we will need to assume that $\T_4$ does not contain too \emph{few} edges, which will lead to a case distinction in the proof of Theorem~\ref{thm:main}.

Recall that by definition of
$\T_4$, every triangle in $\T_4$ sends more than $8(t_4-1)$ edges to the other
triangles of $\T_4$, which should be seen as something like a `minimum degree'
condition. Imposing the further condition $e(\T_4)\ge 8\binom{t_4}{2}+10t_4-27$
has the consequence that there must exist some pairs of triangles in $\T_4$
which are connected by nine edges; the combination of the two features makes
$\T_4$ an exceptionally good place for construction of complex rotations. Our
aim is to take advantage of this in order to provide a good bound on
$e(\T_4\cup\M\cup\I)$.

Unfortunately, this will mean that we can no longer use Lemma~\ref{clm:InTupper}
to provide us with our upper bounds on $e(\T_1\cup\M,\T_4)$ and $e(\T_2\cup
\I,\T_4)$, and we will be forced to use instead Lemma~\ref{clm:TiTj}. This lemma
only gives bounds on $e(\T_1,\T_4)$ when $t_1\neq 1$, and on $e(\T_2,\T_4)$ when
$t_2\neq 1$, which causes a problem that we must now deal with. Consequently, if either
$t_1=1$ and the triangle in $\T_1$ sends more than $7t_4+18$ edges to $\T_4$, or $t_2=1$
and the triangle in $\T_2$ sends more than $8t_4$ edges to $\T_4$, or both, we
will have to handle these one or two exceptional triangles along with $\T_4$.
Fortunately, this adds only a slight complication.

Let \index{$\T_5$}$\T_5$ contain all triangles of $\T_4$, together with $\T_1$ if $t_1=1$ and
$e(\T_1,\T_4)>7t_4+18$, and with $\T_2$ if $t_2=1$ and $e(\T_2,\T_4)>8t_4$. Let
$t_5=|\T_5|$. That is, we have $t_4\le t_5\le t_4+2$.

First, the fact that every triangle in $\T_4$ sends more than $8(t_4-1)$ edges
to the other triangles of $\T_4$ makes $\T_4$ well connected. The following
definition makes this precise.

\begin{definition}[connect, favour]\index{connect}\index{favour}
  Given two triangles~$T$ and~$T'$, we say that a third
  triangle~$T''$ \emph{connects} $T$ to $T'$, or that there is a \emph{connection} from $T$ to $T'$ via $T''$, if one of the following two
  conditions holds.
  \begin{enumerate}[label=\rom]
    \item There are at least $8$ edges from $T''$ to both $T$ and $T'$, or
    \item There are $9$ edges from $T''$ to $T$, and at least $7$ from
      $T''$ to $T'$.
  \end{enumerate}
  To emphasise that the definition is not symmetric in $T$ and $T'$ we say that
  the connection \emph{favours} $T$ and also write \index{$\conn$}$T\conn T''\conn T'$.
\end{definition}

We show that two triangles in $\T_4$ can be connected in many
different ways.

\begin{lemma}\label{ConnWidg} 
  For any pair of distinct triangles $T$ and $T'$ of $\T_4$, there are at least
  $\frac{1}{12}(t_4-2)$ triangles $T''\in\T_4$ with
  $T\conn T''\conn T'$.
\end{lemma}
\begin{proof}
  Suppose first that there are at least $\frac{7}{12}(t_4-2)$ triangles in
  $\T_4\setminus\{T,T'\}$ which send $8$ or more edges to $T$. By
  the definition of $\T_4$ we have $e(T',\T_4\setminus\{T'\})>8(t_4-1)$, 
  and so in particular there are at most
  $\frac{1}{2}(t_4-2)$ triangles of $\T_4\setminus\{T,T'\}$ which send seven or less edges
  to $T'$. Hence at least $\frac{1}{12}(t_4-2)$ triangles of $\T_4$ must
  send at least eight edges to both $T$ and $T'$, as required.

  If on the other hand there are less than $\frac{7}{12}(t_4-2)$ triangles
  in $\T_4\setminus\{T,T'\}$ sending eight or more edges to $T$, then there are more than
  $\frac{5}{12}(t_4-2)$ triangles of $\T_4\setminus\{T,T'\}$ which send at most seven edges to
  $T$. Hence, since $e(T,\T_4\setminus\{T,T'\})\ge 8(t_4-2)$,
  there must also be more than $\frac{5}{12}(t_4-2)$ triangles in $\T_4$ which
  send nine edges to $T$. Again by definition of $\T_4$, of these, at least
  $\frac{1}{12}(t_4-2)$ must also send seven or more edges to $T'$, as
  required.
\end{proof}

Our next Lemma now uses this observation to obtain structural information
about $\T_5\cup\M\cup\I$. Here we need that $t_4$ is
sufficiently large.

\begin{lemma}\label{T4rotate}
  Provided that $e(\T_4)\geq 8\binom{t_4}{2}+10t_4-27$ and $t_4\geq 176$,
  there is no set of vertex-disjoint triangles induced by $V(\T_5\cup\M\cup\I)$ which covers
  three or more vertices of $\M\cup\I$.
\end{lemma}
\begin{proof} 
Suppose the statement is false, that is, there exists a set of
vertex-disjoint triangles in $\T_5\cup\M\cup\I$ which covers three or more
vertices of $\M\cup\I$.

Then we have one of the following three Situations.
\begin{enumerate}[label=\rom]
 \item\label{T4rotate:i} 
   There are three triangles which each consist of a vertex of $\M\cup\I$ and
    an edge in $\T_5$.
  \item\label{T4rotate:ii}
    There is one such triangle and one triangle consisting
    of an edge in $\M\cup\I$ and a vertex of $\T_5$.
  \item\label{T4rotate:iii}
    There are two triangles of the latter type.
\end{enumerate}

We denote the set of these two or three vertex disjoint triangles by $\cS$
and call them \emph{extra triangles}. We denote the set of vertices in
these triangles that are in~$\T_5$ by~$Z$. Observe that $|Z|\le 6$ and
therefore~$Z$ meets at most six triangles in~$\T_5$ which we denote by
$\cZ_5\subset \T_5$.

The idea now is as follows. If we are in Case~\ref{T4rotate:i} and $\cZ_5$
contained only two triangles we immediately arrived at a contradiction since we
could replace~$\cZ_5$ by~$\cS$ and obtain a triangle factor with one triangle
more than~$\T$. Similarly, if we are in Case~\ref{T4rotate:ii}
or~\ref{T4rotate:iii} we cannot have $|\cZ_5|=1$. These two observations
together mean that we cannot have $|\cZ_5|<|\cS|$. 
We will show in the
following that, by way of a sequence of rotations, we can turn any
configuration of $\cZ_5$ into a configuration resembling such a situation and hence
arrive at a contradiction.

More precisely, we shall proceed as follows. Let $V_5$ be the set of
vertices covered by $\cZ_5\cup\cS$. Throughout our process we shall keep
track of a set of \emph{new triangles}~$\cN'$ and a set of \emph{deleted
  triangles}~$\cD'$ such that 
\begin{equation}
\label{eq:T4rotate:cNcD}
  \cN'\cap\T_5=\emptyset \qquad\text{and}\qquad
  \cD'\subset\T_5 \qquad\text{and}\qquad
  |\cD'|=|\cN'|\le 29
  \,.
\end{equation}
In the beginning we set $\cD'=\cN'=\emptyset$. In each
step, we will consider the set of triangles
\begin{equation*} 
  \T'_5:=(\T_5\setminus\cD')\cup\cN'\cup\cS\,. 
\end{equation*}
It will not be true in general throughout the process that~$\T'_5$ is a triangle factor (observe that this for example fails initially). On the other hand we will always have that
\begin{equation}
\label{eq:T4rotate:cover}
  \text{each vertex of~$V_5$ is covered either by one or by two triangles
    of $\T'_5$}\,.
\end{equation}
 We will denote the set of vertices covered by two triangles
by~$Z'$ and call them the \emph{marked vertices}. We let~$\cZ'_5$ be the
set of those triangles of~$\T_5$ which contain a marked vertex and we call
these triangles the \emph{marked triangles}. Note that in the beginning we
have $Z'=Z$ and $\cZ'_5=\cZ_5$. Further, in each step we will have that 
\begin{equation}
\label{eq:T4rotate:mark}
  \text{every marked vertex is contained in a triangle of $\T_5\setminus
  \cD'$}\,,
\end{equation}
which implies that in each step $\T'_5\setminus\cZ'_5$ is a triangle factor
of size $|\T_5|-|\cZ'_5|+|\cS|$ by~\eqref{eq:T4rotate:cNcD}.

In each step we will now perform a rotation by adding
some vertex disjoint triangles in
$G[V_5]$ to the set of new triangles~$\cN'$, and deleting as many triangles
from~$\T'_5\cap \T_5$, i.e., we will add these triangles to the set of deleted
triangles~$\cD'$. We will have three preparation steps (Preparation 1--3) and three main rotation types (Type 1--3).
No step will change the size of~$Z'$ and 
\begin{equation}
\label{eq:T4rotate:dec}
  \text{each rotation of Type 1, 2 or 3 will decrease the size of~$\cZ'_5$}\,.
\end{equation}
We will stop when $|\cZ'_5|<|\cS|$, since
then $\T'_5 \setminus\cZ'_5$ is a triangle factor with more triangles than
$\T_5$, a contradiction.

It remains to construct~$\cZ'_5$ with these properties. We will first carry out three preparatory steps: roughly, these consist of locating two disjoint copies of $K_6$ in $\T_4$ (Preparation~\ref{prep:K6}) and showing that we can `move' $Z'$ to $\T_4$ (which is useful because Lemma~\ref{ConnWidg} then applies), in Preparations~\ref{prep:t1} and~\ref{preround3}. After this we have either two, three, or six marked vertices in $\T_4$. Our next aim is to `move around' these vertices within $\T_4$ such that they are contained in one, one or two (respectively) triangles of $\T_4$. Achieving this immediately gives us $|\cZ'_5|<|\cS|$, which is what we want. To do this we make use of our main rotation Types~\ref{type:T4rotate:1},~\ref{type:T4rotate:2} and~\ref{type:T4rotate:3}. We will now give details of the Preparation steps and the main rotation Types.

\begin{prepround}\label{prep:K6}
  There are two disjoint pairs $(T_1,T'_1)$ and $(T_2,T'_2)$ of
  triangles in~$\T_4$ which do not meet~$Z$ and are such that $V(T_1)\cup
  V(T'_1)$ and $V(T_2)\cup V(T'_2)$ each induce a copy of~$K_6$ in~$G$.
  We set $\cK_6:=\{T_1,T'_1,T_2,T'_2\}$ and call $(T_1,T'_1)$ and
  $(T_2,T'_2)$ the $K_6$-copies of~$\cK_6$.
\end{prepround}

To see this, let $H=(\T_4,E_H)$ be the auxiliary graph with
edges exactly between those triangles $T,T'\in\T_4$ which are
connected by nine edges.  Since $e(\T_4)\ge 8\binom{t_4}{2}+10t_4-27$ by
assumption and
\begin{equation*}
  e(\T_4)\le 9
  e(H)+8\big(\tbinom{t_4}{2}-e(H)\big)+3t_4=e(H)+8\tbinom{t_4}{2}+3t_4,
\end{equation*} 
we conclude that $e(H)\ge7t_4-27$. Since $\max\big(7(t_4-7)+\binom{7}{2},\binom{15}{2}\big)=7t_4-28<e(H)$, we can apply
Theorem~\ref{thm:ErdGal} to~$H$ and infer that there are at least eight
independent edges in~$H$, and hence at least two independent edges in $H$ which
do not meet $\cZ_5$. These two edges give us the pairs $(T_1,T'_1)$ and $(T_2,T'_2)$.

\begin{prepround}\label{prep:t1} Suppose that $\{uvw\}=\T_1\subset\T_5$. We distinguish four cases.

{\sl Case~1:} In the case when $Z\cap \T_1=\emptyset$ we do not do anything.

{\sl Case~2:} If $Z\cap \T_1=\{u\}$, then we consider the edges
between $uvw$ and $\T_4$. Because there are in total at least $7t_4+19$ such
edges (recall that this was the condition for inclusion of $\T_1$ in $\T_5$), in particular there must be at least ten triangles of $\T_4$ to which
$uvw$ sends more than seven edges. Now at most $9$ of these triangles are in
$\cZ_5\cup\cK_6$, as no triangle of $\T_4$ covers $u\in Z$. Therefore there is a triangle
$xyz\in\T_4\setminus(\cZ_5\cup\cK_6)$ to which $uvw$ sends at least eight edges.
Thus $vw$ has a common neighbour, say $x$, in $xyz$. We add $xvw$ to the set of new
triangles $\cN'$, and $uvw$ to the set of deleted triangles $\cD'$. The upshot
is that $u$ is no longer marked, but $x$, which lies in a triangle of
$\T_4\setminus(\cZ_5\cup\cK_6)$, is.

{\sl Case~3:} If $Z\cap \T_1=\{u,v\}$ then we work similarly: again, there is a triangle
$xyz\in\T_4\setminus(\cZ_5\cup\cK_6)$ to which $uvw$ sends at least eight edges,
and we may assume $w$ is adjacent to both $x$ and $y$. We add $xyw$ to $\cN'$,
and $uvw$ to $\cD'$. The result is that $u$ and $v$ are no longer marked, but
$x$ and $y$ are.

{\sl Case~4:} If $Z\cap \T_1=\T_1$, we may again simply ignore $\T_1$ (keeping the vertices of $\T_1$ still marked). The only possibility is that we are in
situation~\ref{T4rotate:i} or in situation~\ref{T4rotate:ii}.
We rule out situation~\ref{T4rotate:ii} as follows. If $auv$
and $bcw$ are the two triangles from
situation~\ref{T4rotate:ii} then replacing $uvw\in \T_1$ by
$auv$ and $bcw$ is an improving rotation, a contradiction.
\end{prepround}

\begin{prepround}\label{preround3} Suppose that
$\{u'v'w'\}=\T_2\subset\T_5$. We behave exactly as above, which we may do because $e(\T_2,\T_5)>8t_4\ge 7t_4+19$. The first inequality is by definition of $\T_5$, and the second by the assumption $t_4\ge 176$.
\end{prepround}

Before describing the main rotation Types, let us briefly recap the current situation. We have a set $Z'$ of marked vertices, which contains either six, three or two vertices (in Situation~\ref{T4rotate:i},~\ref{T4rotate:ii} or~\ref{T4rotate:iii} respectively). If $|\T_1|=|\T_2|=1$ and there are six marked vertices are in $\T_1\cup\T_2$ then removing the two triangles $\T_1\cup\T_2$ from $\T$ and adding the three triangles $\cS$ is an improving rotation, which is a contradiction. It follows that either all the marked vertices are in $\T_4$, or we have six marked vertices, of which either three are in the unique triangle of $\T_1$ or three are in the unique triangle of $\T_2$, and the remaining three are in $\T_4$. We have a set of at most two deleted triangles $\cD'$ (at most one from each of Preparation~\ref{prep:t1} and~\ref{preround3}) none of which are in $\T_4$. Finally, we have a set $\cK_6$ consisting of four triangles of $\T_4$ which span two disjoint copies of $K_6$, none of whose vertices are marked.

We now describe the main rotation Types.

\begin{type}
\label{type:T4rotate:1}
  Suppose that $\big|\T_4\cap(\cD'\cup\cZ'_5)\big|\le 11$, and that there are two triangles $uvw$ and $u'v'w'$ of $\cZ'_5$, such that $Z'\cap\{u,v,w,u',v',w'\}=\{u,u'\}$. We can add two triangles to $\cD'$, neither in $\cK_6$, and two triangles to $\cN'$, and obtain $\big|Z'\cap\{u,v,w\}\big|=2$.
\end{type}

This type of rotation can be constructed for the following
reason. By Lemma~\ref{ConnWidg} there are $\tfrac{1}{12}(t_4-2)>14$ triangles
$xyz$ in $\T_4$ such that $uvw\conn xyz\conn u'v'w'$. Of these, at most $4$
are in $\cK_6$, and, because $xyz$ is neither $uvw$ nor $u'v'w'$, at most
$9$ are in $\T_4\cap(\cD'\cup\cZ'_5\big)$. It follows that we may choose
$xyz$ in $\T_4\setminus(\cD'\cup\cK_6\cup\cZ'_5)$ such that $uvw\conn
xyz\conn u'v'w'$.  Because of this connection, at least one vertex of
$xyz$, say~$x$, is adjacent to both $v'$ and $w'$. In addition, because the
connection favours $uvw$, at least two vertices of $uvw$ are adjacent
to~$y$ and~$z$.  In particular, one vertex of $uvw$ different from~$u$,
say~$v$, forms a triangle with~$y$ and~$z$. Now we can rotate by adding the
triangles $u'v'w'$ and $xyz$ to the set~$\cD'$ of deleted triangles and the
triangles $v'w'x$ and $vyz$ to the set~$\cN'$ of new triangles.  Observe
that this rotation satisfies~\eqref{eq:T4rotate:cNcD} and
\eqref{eq:T4rotate:cover}. Further, it removes~$u'$ from~$Z'$ and $u'v'w'$
from~$\cZ'_5$ and adds~$v$ to~$Z'$ and no new triangle
to~$\cZ'_5$. Hence~\eqref{eq:T4rotate:mark} and~\eqref{eq:T4rotate:dec} are
also satisfied.

\begin{type}
\label{type:T4rotate:2}
  Suppose that $\big|\T_4\cap(\cD'\cup\cK_6\cup\cZ'_5)\big|\le 15$, that at least one of the copies of $K_6$ in $\cK_6$ does not meet $\cD'$, and that there are two triangles $uvw$ and $u'v'w'$ of $\cZ'_5$, such that $Z'\cap\{u,v,w,u',v',w'\}=\{u,v,u'\}$. We can add five triangles to $\cD'$, exactly two of which are in $\cK_6$, and five triangles to $\cN'$, and obtain $Z'\cap\{u,v,w\}=\{u,v,w\}$.
\end{type}

Let the copy of $K_6$ in $\cK_6$ not meeting $\cD'$ be on the triangles $abc,def$ of $\T_4$. By Lemma~\ref{ConnWidg} there are at least $\tfrac{1}{12}(t_4-2)$ triangles
$xyz$ in $\T_4$ with
$uvw \conn xyz \conn abc$. Since $xyz$ is neither $uvw$ nor $abc$, by assumption there are at least two choices of $xyz\not\in \cD'\cup \cK_6\cup \cZ'_5$. We fix one. Similarly, by Lemma~\ref{ConnWidg} there is a choice of triangle $x'y'z'$ in
$\T_4\setminus(\cD'\cup\cK_6\cup\cZ'_5\cup\{uvw\})$ with 
$u'v'w'\conn x'y'z'\conn def$ (see also Figure~\ref{fig:secconst}).  Because of
the second connection, at least
one vertex of $x'y'z'$, say $x'$, is a common neighbour of $v'w'$, and at
least one vertex of $def$, say~$e$, is a common neighbour of $y'z'$. We
conclude that $x'v'w'$ and $ey'z'$ are triangles.

\begin{figure}[utilise]
\centering
\psfrag{u}{$u$}
\psfrag{v}{$v$}
\psfrag{w}{$w$}
\psfrag{u'}{$u'$}
\psfrag{v'}{$v'$}
\psfrag{w'}{$w'$}
\psfrag{x}{$x$}
\psfrag{y}{$y$}
\psfrag{z}{$z$}
\psfrag{x'}{$x'$}
\psfrag{y'}{$y'$}
\psfrag{z'}{$z'$}
\psfrag{a}{$a$}
\psfrag{b}{$b$}
\psfrag{c}{$c$}
\psfrag{d}{$d$}
\psfrag{e}{$e$}
\psfrag{f}{$f$}
\includegraphics[height=60mm]{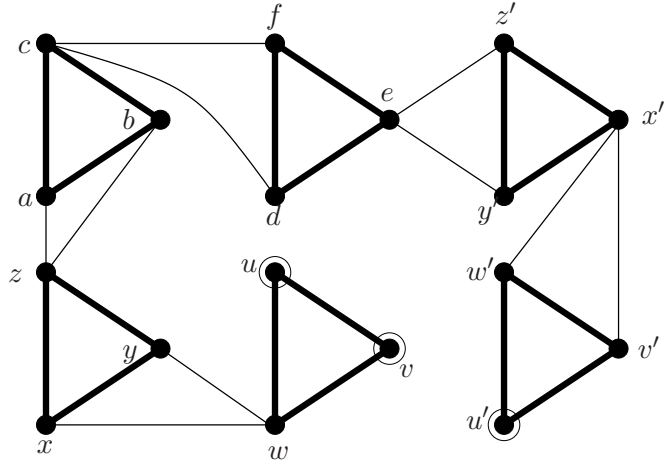}
\caption{The second rotation type.}
\label{fig:secconst}
\end{figure}
 
Now we distinguish two possibilities concerning the connection between
$uvw$ and $abc$. First, there are at least eight edges from $xyz$ to both
$uvw$ and $abc$. In this case, we are guaranteed that at least two
vertices, say $x$ and $y$, of $xyz$ are adjacent to $w$, and at least two
vertices, say $a$ and $b$, of $abc$ are adjacent to $z$. Hence $xyw$, $abz$
and $cdf$ are triangles in~$G$. Second, since the connection favours
$uvw$, from $xyz$ there are nine edges to $uvw$ and seven to $abc$. Some
vertex of $xyz$, say $z$, is adjacent to both, $a$ and $b$. Therefore,
again, $abz$, $xyw$, and $cdf$ are triangles.

Accordingly we can rotate by adding $x'v'w'$, $ey'z'$, $cdf$, $abz$, and
$xyw$ to~$\cN'$, and $u'v'w'$, $x'y'z'$, $def$, $abc$, and $xyz$
to~$\cD'$. This deletes~$u'$ from~$Z'$ and hence $u'v'w'$ from~$\cZ'_5$, it
adds~$w$ to~$Z'$ and no triangle to~$\cZ'_5$. Hence, as can easily be
checked, this rotation satisfies~\eqref{eq:T4rotate:cNcD},
\eqref{eq:T4rotate:cover}, \eqref{eq:T4rotate:mark},
and~\eqref{eq:T4rotate:dec}.

\begin{type}
\label{type:T4rotate:3}
  Suppose that $\big|\T_4\cap(\cD'\cup\cK_6\cup\cZ'_5)\big|\le 12$ and $\cK_6\cap\cD'=\emptyset$, and that there are three triangles $uvw$, $u'v'w'$, and
  $u''v''w''$ of $\cZ_4'\cap \T_4$ such that $Z'=\{u,v,u',v',u'',v''\}$. Then we can add at most ten triangles to $\cD'$ and ten triangles to $\cN'$, and obtain $Z'=\{u,v,w,u',v',w'\}$.
\end{type}

Let the two copies of $K_6$ in $\cK_6$ be $(abc,def)$ and $(a'b'c',d'e'f')$.
We apply Lemma~\ref{ConnWidg} five times to obtain the following 
connections which avoid each other and whose connecting triangles are from
$\T_4\setminus(\cD'\cup\cK_6\cup\cZ'_5)$:
$uvw \conn xyz \conn abc$,
$u'v'w'\conn x'y'z'\conn def$, 
$u''v''w'' \conn x''y''z''\conn a'b'c'$, 
$abc\conn a''b''c''\conn def$, and 
$d'e'f'\conn d''e''f''\conn a''b''c''$. Observe that this is possible since at each application Lemma~\ref{ConnWidg} guarantees at least $15$ connecting triangles in $\T_4$, while at each application there are by assumption at most $12-2=10$ triangles of $\cD'\cup\cK_6\cup\cZ'_5$ to avoid (since two triangles from this set are being connected and are thus automatically avoided), together with the at most four previously determined connecting triangles which must also be avoided.

\begin{figure}[htbp]
\centering
\psfrag{u}{$u$}
\psfrag{v}{$v$}
\psfrag{w}{$w$}
\psfrag{u'}{$u'$}
\psfrag{v'}{$v'$}
\psfrag{w'}{$w'$}
\psfrag{u''}{$u''$}
\psfrag{v''}{$v''$}
\psfrag{w''}{$w''$}
\psfrag{x}{$x$}
\psfrag{y}{$y$}
\psfrag{z}{$z$}
\psfrag{x'}{$x'$}
\psfrag{y'}{$y'$}
\psfrag{z'}{$z'$}
\psfrag{x''}{$x''$}
\psfrag{y''}{$y''$}
\psfrag{z''}{$z''$}
\psfrag{a}{$a$}
\psfrag{b}{$b$}
\psfrag{c}{$c$}
\psfrag{a'}{$a'$}
\psfrag{b'}{$b'$}
\psfrag{c'}{$c'$}
\psfrag{a''}{$a''$}
\psfrag{b''}{$b''$}
\psfrag{c''}{$c''$}
\psfrag{d}{$d$}
\psfrag{e}{$e$}
\psfrag{f}{$f$}
\psfrag{d'}{$d'$}
\psfrag{e'}{$e'$}
\psfrag{f'}{$f'$}
\psfrag{d''}{$d''$}
\psfrag{e''}{$e''$}
\psfrag{f''}{$f''$}
\includegraphics[width=360pt]{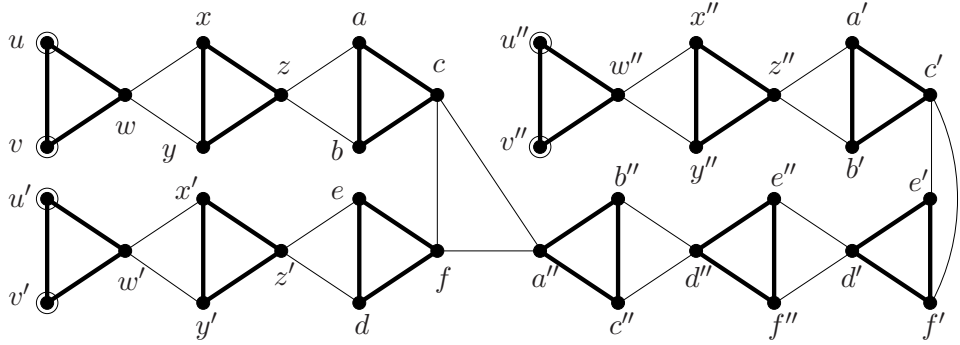}
\caption{The third rotation type.}
\label{fig:thirdconst}
\end{figure}

Arguing similarly as before, these connections guarantee, without
loss of generality, the triangles $wxy$, $w'x'y'$,
$w''x''y''$, $zab$, $z'de$, and $z''a'b'$. Since $cf$ belongs to a~$K_6$,
it is an edge, and because of the connection $abc\conn
a''b''c''\conn def$, there is a vertex, say $a''$,
of $a''b''c''$ which is adjacent to both $c$ and $f$, and
so $cfa''$ is a triangle of~$G$. 
Finally, using the connection $d'e'f'\conn d''e''f''\conn a''b''c''$,
we can find a common neighbour, say $d''$, of $b''c''$ in $d''e''f''$, and
and a common neighbour, say $d'$, of $e''f''$ in $d'e'f'$. Hence, 
$b''c''d''$, $d'e''f''$, and $c'e'f'$ are triangles of~$G$. 
See Figure~\ref{fig:thirdconst}.

We conclude that we can rotate by  adding the ten triangles $wxy$, $w'x'y'$, $w''x''y''$,
$zab$, $z'de$, $z''a'b'$, $cfa''$, $b''c''d''$, $d'e''f''$,  $c'e'f'$
to~$\cN'$ and adding the ten triangles $u''v''w''$,
$xyz$, $x'y'z'$, $x''y''z''$, $abc$, $def$, $a'b'c'$, $d'e'f'$,
$a''b''c''$, $d''e''f''$ to~$\cD'$.
This removes~$u''$ and~$v''$ from~$Z'$ and hence
$u''v''w''$ from~$\cZ'_5$, and adds~$w$ and~$w'$ to~$Z'$ and no new triangle
to~$\cZ'_5$.
Again, it is easy to check that this rotation satisfies~\eqref{eq:T4rotate:cNcD},
\eqref{eq:T4rotate:cover}, \eqref{eq:T4rotate:mark},
and~\eqref{eq:T4rotate:dec}.

We now explain how we apply these rotation Types. If we started in Situation~\ref{T4rotate:iii}, then $Z'$ consists of two vertices in $\T_4$. These two vertices are in distinct triangles of $\T_4$ (since $|\cZ'_5|\ge|\cS|=2$). We apply rotation Type~\ref{type:T4rotate:1} to the two triangles of $\cZ'_5$, which we can do since $\big|\T_4\cap(\cD'\cup\cZ'_5)\big|=2$. This adds two triangles to each of $\cD'$ and $\cN'$, and reduces $\cZ'_5$ to one triangle. So we have $|\cZ'_5|<|\cS|$ and we are done.

If we started in Situation~\ref{T4rotate:ii}, then $Z'$ consists of three vertices in $\T_4$. These may either lie in two or three triangles of $\T_4$ (since $|\cZ'_5|\ge|\cS|=2$). In the latter case we apply rotation Type~\ref{type:T4rotate:1} to two of the triangles of $\cZ'_5$ (which we may do for the same reason as above), which reduces $\cZ'_5$ to two triangles. Now since $\cZ'_5$ has two triangles, so one contains two vertices of $Z'$ and the other contains one. We apply rotation Type~\ref{type:T4rotate:2} to the two triangles of $\cZ'_5$, which we may do since $\big|\T_4\cap(\cD'\cup\cK_6\cup\cZ'_5)\big|\le 2+4+2=8$, and obtain $|\cZ'_5|=1<|\cS|$: we are done.

Finally, suppose we started in Situation~\ref{T4rotate:i}. Now $Z'$ contains six vertices and $|\cS|=3$. These cannot all lie in two triangles of $\T_5$, since otherwise deleting these two triangles and adding $\cS$ to $\T$ is an improving rotation. If three vertices of $Z'$ lie in one triangle of $\T_5$, and the remaining three lie in either two or three triangles (which must be in $\T_4$) then we apply the identical rotation strategy as in Situation~\ref{T4rotate:ii}. We may have $\cZ'_5$ larger by one than there, but nevertheless the rotations exist. The remaining possibility is that all six vertices of $Z'$ lie in $\T_4$, and no three are contained in any one triangle of $\T_4$. We separate several possibilities.

First, if the six vertices lie in three triangles of $\T_4$, then we apply rotation Type~\ref{type:T4rotate:3}, which we may do since $\big|\T_4\cap(\cD'\cup\cK_6\cup\cZ'_5)\big|=0+4+3=7$, and obtain $|\cZ'_5|=2<|\cS|$: we are done.

If the six vertices lie in five or six triangles of $\T_4$, then we apply rotation Type~\ref{type:T4rotate:1} either once or twice. In the first application we have $\big|\T_4\cap(\cD'\cup\cZ'_5)\big|\le 0+6=6$, while in the second application (if we apply it twice) $\big|\T_4\cap(\cD'\cup\cZ'_5)\big|\le 2+5=7$, so we are permitted to do this. We add either two or four triangles to each of $\cD'$ and $\cN'$, and reduce $\cZ'_5$ to four triangles, which is our final case.

The final case we have to handle is that $\cZ'_5$ contains four triangles, of which two contain two vertices of $Z'$ each and two contain one each. We apply rotation Type~\ref{type:T4rotate:2} twice. In the first application we have $\big|\T_4\cap(\cD'\cup\cK_6\cup\cZ'_5)\big|\le 4+4+4=12$, while in the second we have $\big|\T_4\cap(\cD'\cup\cK_6\cup\cZ'_5)\big|\le 9+2+3=14$, since the first application adds five triangles to $\cD'$, two of which are in $\cK_6$, and therefore we can indeed construct these rotations. After the second application of rotation Type~\ref{type:T4rotate:2} we have $|\cZ'_5|=2<|\cS|$ and we are done.
\end{proof}

We are able to convert the structural information provided by Lemma~\ref{T4rotate} into an upper bound on $e(\T_5\cup\M\cup\I)$. We need to define the following function.
\begin{equation}\label{eq:defpprime}
  p(h,a):=
  \begin{cases}
    a(h-a)+\binom{h-2a}{2}+6h & 2a\le h< 9a\;,\\
    (a-2)(h-a+2)+\binom{h-2a+4}{2} & 9a\le h\;.
  \end{cases}
\end{equation}

The connection between this function and $e(\T_5\cup\M\cup\I)$ is provided by the following lemma, whose proof we defer to Section~\ref{sec:fewtriangles}.

\begin{lemma}\label{lem:T4upperMG}
There exists $\kappa_0$ such that the following holds. Let $H$ be a graph of order $h\ge\kappa_0$. Suppose that $A$ is a subset of $V(H)$, with $3\le|A|\le h/2$ and the property that there is no set
of vertex-disjoint triangles in $H$ which covers three or more
vertices of $\A$.
Then $e(H)\le p(h, |\A|)$.
\end{lemma}

Putting this lemma together with Lemma~\ref{T4rotate} allows us to strengthen the bound~\eqref{T4}. This is the missing ingredient for the proof of Theorem~\ref{thm:main}.

\begin{lemma}\label{clm:T4upper}
  There exists $\kappa_0$ such that the following holds. Provided that $e(\T_4)\geq 8\binom{t_4}{2}+10t_4-27$ and $t_4\geq \max\big(176,\kappa_0,\tfrac{2m+i}{3}\big)$ we
  have
  \[e(\T_5\cup\M\cup\I)\leq p(3t_5+2m+i,2m+i)\,.\]
\end{lemma}
\begin{proof}
  Suppose that $e(\T_4)\geq8\binom{t_4}{2}+10t_4-27$ and $t_4\geq\max(176,\kappa_0)$. By Lemma~\ref{T4rotate} there is no set of vertex-disjoint triangles induced by $V(\T_5\cup\M\cup\I)$ which covers three or more vertices of $\M\cup\I$. We then apply Lemma~\ref{lem:T4upperMG} to $G[\T_5\cup\M\cup\I]$, with the partition into $\T_5$ and $\M\cup\I$. We conclude that the number of edges in this graph is at most $p(3t_5+2m+i,2m+i)$ as desired.
\end{proof}

\section{Proof of Theorem~\ref{thm:main}}\label{sec:proof}

  We are now in a position to prove Theorem~\ref{thm:main}. The basic idea is the same as for the proof of Lemma~\ref{lem:tri1}. We assume Setup~\ref{setup}, and put together our various upper bounds on edges between parts to obtain a function of six variables (the sizes of the six parts) which upper bounds the number of edges in $G$. We then show that this function is maximised, subject to the constraints $t_1+t_2+t_3+t_4=k$ and $2m+i=n-3k$, by $e\big(E_i(n,k)\big)$ for some $i\in[4]$.
  
  A small problem with this strategy is that Lemma~\ref{clm:T4upper}, which we would like to use to provide one of our upper bounds, only applies if $e(\T_4)\ge 8\binom{t_4}{2}+10t_4-27$. We therefore have to handle the case that $e(\T_4)\le8\binom{t_4}{2}+10t_4-28$ separately. We need to define a function, which we obtain as follows. Summing the bounds in
Lemmas~\ref{clm:NRupper} and~\ref{clm:InTupper}, together with
the assumption $e(\T_4)\le 8\binom{t_4}{2}+10t_4-28$, we see that the following function bounds above $e(G)$.
\begin{equation}\label{eq:Defngsmall}
  \begin{split}\index{$g_s(t_1,t_2,t_3,t_4,m,i)$}
    g_s(t_1,t_2,t_3,t_4,m,i):= & f(t_1,t_2,t_3,t_4,m,i) +im+m^2\\
    &+(3+3m)t_4+(2+i)t_4+8\binom{t_4}{2}+10t_4-28\;.
  \end{split}
\end{equation}

The maximisation of $g_s(t_1,t_2,t_3,t_4,m,i)$ subject to
$t_1+t_2+t_3+t_4=k$ and $2m+i=n-3k$ is a matter of calculation which we defer to Appendix~\ref{sec:max}.

\begin{lemma}\label{lem:maxgsmall} If $n\ge 8406$  and $(\ttt_1,\ttt_2,\ttt_3,\ttt_4,\mmm,\iii)\in F(n,k)$ then
  \begin{equation*}
    g_s(\ttt_1,\ttt_2,\ttt_3,\ttt_4,\mmm,\iii)\le
    \max_{j\in[4]}e\big(E_j(n,k)\big)\,.
  \end{equation*}
\end{lemma}

The final function, $g_\ell$, that we need to define, which we will show bounds above $e(G)$ provided that $e(\T_4)\ge 8\binom{t_4}{2}+10t_4-27$, is a little more complicated. Its definition is as follows.

If $t_4<\max\big(176,\kappa_0,\tfrac{2m+i}{3}\big)$, then $g_\ell(t_1,t_2,t_3,t_4,m,i)$ is
defined by\index{$g_\ell(t_1,t_2,t_3,t_4,m,i)$}
\begin{equation}\label{defglsmt4}
  g_\ell:=f(t_1,t_2,t_3,t_4,m,i)+im+m^2+(3+3m)t_4+
  (2+i)t_4+\binom{3t_4}{2}\,.
\end{equation}

If $t_4\ge\max\big(176,\kappa_0,\tfrac{2m+i}{3}\big)$ and $t_1\neq 1$ then we set
\begin{equation}\label{defglnorm}
  g_\ell(t_1,t_2,t_3,t_4,m,i):=f(t_1,t_2,t_3,t_4,m,i)+p(3t_4+2m+i,2m+i)\,.
\end{equation}

If $t_4\ge\max\big(176,\kappa_0,\tfrac{2m+i}{3}\big)$ and $t_1=1$ then we set
\begin{equation}\label{defglt11}
  g_\ell(t_1,t_2,t_3,t_4,m,i):=f(t_1,t_2,t_3,t_4,m,i)+p(3t_4+2m+i,2m+i)+20\,.
\end{equation}

The following lemma, whose proof we defer to Appendix~\ref{sec:max}, states that $g_\ell$
is upper bounded as desired.
\begin{lemma}\label{lem:maxgl}
\setcounter{lemma23counter}{\value{theorem}}
  Let $n\ge\max(4\cdot 10^4,900\kappa_0)$ and $k\in\mathbb N$ be given. If $n\le  5k+8$, we have \[\max_{(\ttt_1,\ttt_2,\ttt_3,\ttt_4,\mmm,\iii)\in
  F(n,k)}g_\ell(\ttt_1,\ttt_2,\ttt_3,\ttt_4,\mmm,\iii)\le
  \max_{j\in[4]}e\big(E_j(n,k)\big)\,.\]
\end{lemma}

The proof of Theorem~\ref{thm:main} now amounts to verification that the
functions $g_s$ and $g_\ell$ indeed upper bound $e(G)$ as required.

\begin{proof}[Proof of Theorem~\ref{thm:main}]
  Given $n$ and $k$, let $G$ be an $n$-vertex graph which does not contain
  $(k+1)\times K_3$. Further, assume that $n\ge \max(4\cdot 10^4,900\kappa_0)$. We assume
  $G$ is decomposed as in Setup~\ref{setup}.
  
  If $n>5k+8$, then by Lemma~\ref{lem:tri1} we have
  \[e(G)\le e\big(E_1(n,k)\big)\,,\]
  so we may now assume that $n\le 5k+8$.
  
  If $e(\T_4)\le 8\binom{t_4}{2}+10t_4-28$, then our situation is exactly as
  in~\eqref{eq:Defngsmall}, i.e., by Lemma~\ref{lem:maxgsmall}
  we have
  \[e(G)\le
  g_s(t_1,t_2,t_3,t_4,m,i)\le\max_{j\in[4]}e\big(E_j(n,k)\big)\;,\] which completes
  the proof in this case.
  
  If on the other hand $e(\T_4)\ge 8\binom{t_4}{2}+10t_4-27$, we have the
  following fact.  
  \begin{AuxiliarySubClaimM}\label{fac:glworks}
    If $e(\T_4)\ge 8\binom{t_4}{2}+10t_4-27$ then there exist $c_1,c_2\in\{0,1\}$ such that we have
    \[e(G)\le g_\ell(t_1-c_1,t_2-c_2,t_3,t_4+c_1+c_2,m,i)\,.\]
    Furthermore, we have $t_1-c_1\ge 0$ and $t_2-c_2\ge 0$.
  \end{AuxiliarySubClaimM}
  \begin{factproof}[Proof of Claim~\ref{fac:glworks}] 
  We distinguish five cases.
  
    {\sl Case~1:}  $t_4< \max\big(176,\kappa_0,\tfrac{2m+i}{3}\big)$.
    
    We take $c_1=c_2:=0$, and sum the
    bounds~\eqref{eq:lpd} and~\ref{I}--\ref{M} of Lemma~\ref{clm:NRupper}
    together with the trivial bound $e(\T_4)\le\binom{3t_4}{2}$. We obtain
    \[e(G)\le f(t_1,t_2,t_3,t_4,m,i)+im+m^2+(3+3m)t_4+
    (2+i)t_4+\binom{3t_4}{2}\,\]
    and so by~\eqref{defglsmt4} we have $e(G)\le g_\ell(t_1,t_2,t_3,t_4,m,i)$.
    
    {\sl Case~2:} $t_4\ge \max\big(176,\kappa_0,\tfrac{2m+i}{3}\big)$, $e(\T_1,\T_4)\le 7t_1t_4+18$
    and $e(\T_2,\T_4)\le 8t_2t_4$.
    
    We take again $c_1=c_2:=0$. By definition we have $\T_5=\T_4$. We sum
    the bounds~\ref{MT1}--\ref{T3T4} of Lemma~\ref{clm:NRupper}, the bounds~\ref{T2,T1+M} and the
    $j=3$ cases of~\ref{Ti,T1+M} and~\ref{Ti,T1+I} of Lemma~\ref{clm:InTupper},
    the bound $e(\T_4\cup\M\cup\I)\leq p(3t_4+2m+i,2m+i)$
    from Lemma~\ref{clm:T4upper} and the assumed $e(\T_2,\T_4)\le 8t_2t_4$.
    These bounds cover all the edges of $G$ except $e(\T_1,\T_4)$,
    and we have $$e(G)-e(\T_1,\T_4)\le f(t_1,t_2,t_3,t_4,m,i)+p(3t_4+2m+i,2m+i)-7t_1t_4\;.$$
    If $t_1\neq 1$, then
    Lemma~\ref{clm:TiTj} part~\ref{T1Ti} gives us that $e(\T_1,\T_4)\le
    7t_1t_4$, and we obtain
    \[e(G)\le
    f(t_1,t_2,t_3,t_4,m,i)+p(3t_4+2m+i,2m+i)\;,\]
    which is in correspondence with~\eqref{defglnorm}. If
    $t_1=1$, then the assumed $e(\T_1,\T_4)\le 7t_1t_4+18$ gives us
    \[e(G)\le
    f(t_1,t_2,t_3,t_4,m,i)+p(3t_4+2m+i,2m+i)+18\;,\]
    and by~\eqref{defglt11} we have $e(G)\le g_\ell(t_1,t_2,t_3,t_4,m,i)$.
    
    {\sl Case~3:} $t_4\ge\max\big(176,\kappa_0,\tfrac{2m+i}{3}\big)$, $e(\T_1,\T_4)\le 7t_1t_4+18$
    and $e(\T_2,\T_4)>8t_2t_4$.
    
    We take $c_1:=0$ and $c_2:=1$. Observe that by Lemma~\ref{clm:TiTj}
    part~\ref{T2Ti} $e(\T_2,\T_4)>8t_2t_4$ implies that $t_2=1$. By definition
    of $\T_5$ we have $\T_5=\T_4\cup\T_2$, and by Lemma~\ref{clm:T4upper} we
    have $e(\T_5\cup\M\cup\I)\le p(3t_4+2m+i+3,2m+i)$.
    
    We use the bounds in parts~\ref{MT1}--\ref{T1} and~\ref{T3T4} of
    Lemma~\ref{clm:NRupper} and the $j=3$ cases
    of parts~\ref{Ti,T1+M} and~\ref{Ti,T1+I} of Lemma~\ref{clm:InTupper}.
    Together with the above bound on $e(\T_5\cup\M\cup\I)$, these bounds cover all the
    edges of $G$ except $e(\T_1,\T_2\cup\T_4)$, and we have 
    \begin{align}
\label{eq:HO1}
\begin{split}
e(G)-e(\T_1,\T_2\cup\T_4)&\le
f(t_1,0,t_3,t_4,m,i)\\
&~~+p(3t_4+2m+i+3,2m+i)-7t_1t_4+8t_2t_3\;\mbox{.}
\end{split}
    \end{align}
    If $t_1\neq 1$, then the $j=2$ and $j=4$ cases of Lemma~\ref{clm:TiTj} part~\ref{T1Ti} yield $e(\T_1,\T_2\cup\T_4)\le
    7t_1(t_2+t_4)=7t_1t_5$, and we obtain from~\eqref{eq:HO1}
    that
     \[e(G)\le
    f(t_1,0,t_3,t_5,m,i)+p(3t_4+2m+i+3,2m+i)\,.\]
    By~\eqref{defglnorm} we have $e(G)\le g_\ell(t_1,0,t_3,t_4+1,m,i)$. If
    $t_1=1$, then we use instead the trivial $e(\T_1,\T_2)\le 9$ and the assumed
    $e(\T_1,\T_4)\le 7t_1t_4+18$ to obtain $e(\T_1,\T_2\cup\T_4)\le 7t_1t_5+20$,
    and hence
    \[e(G)\le
    f(t_1,0,t_3,t_5,m,i)+p(3t_4+2m+i+3,2m+i)+20\,,\]
    and by~\eqref{defglt11} we have $e(G)\le g_\ell(t_1,0,t_3,t_4+1,m,i)$.
    
    {\sl Case~4:} $t_4\ge\max\big(176,\kappa_0,\tfrac{2m+i}{3}\big)$, $e(\T_1,\T_4)>
    7t_1t_4+18$ and $e(\T_2,\T_4)\le 8t_2t_4$.
    
    We take $c_1:=1$ and $c_2:=0$. By Lemma~\ref{clm:TiTj} part~\ref{T1Ti} we have
    $t_1=1$. Thus we have $\T_5=\T_4\cup\T_1$. Summing the bounds in
    parts~\ref{IT2}--\ref{T3T4} of Lemma~\ref{clm:NRupper} and those in part~\ref{T2,T1+M} and the $j=3$ cases
    of parts~\ref{Ti,T1+M} and~\ref{Ti,T1+I} of
    Lemma~\ref{clm:InTupper}, together with the assumed $e(\T_2,\T_4)\le8t_2t_4$ and the bound
    $e(\T_5\cup\M\cup\I)\le p(3t_4+2m+i+3,2m+i)$ from
    Lemma~\ref{clm:T4upper}, we obtain
    \[e(G)\le
    f(0,t_2,t_3,t_5,m,i)+p(3t_4+2m+i+3,2m+i)+20\;.\]
    By~\eqref{defglt11} we have $e(G)\le g_\ell(0,t_2,t_3,t_4+1,m,i)$.
    
    {\sl Case~5:} $t_4\ge\max\big(176,\tfrac{2m+i}{3}\big)$, $e(\T_1,\T_4)>7t_1t_4+18$ and
    $e(\T_2,\T_4)>8t_2t_4$.
    
    We take $c_1=c_2=1$. By Lemma~\ref{clm:TiTj} parts~\ref{T1Ti} and~\ref{T2Ti}
    we have $t_1=t_2=1$, and thus we have $\T_5=\T_1\cup\T_2\cup\T_4$. Summing
    the bounds in part~\ref{T3T4} of Lemma~\ref{clm:NRupper} and the $j=3$ cases
    of parts~\ref{Ti,T1+M} and~\ref{Ti,T1+I} of Lemma~\ref{clm:InTupper},
    together with the bound
    $e(\T_5\cup\M\cup\I)\le p(3t_4+2m+i+6,2m+i)$ from
    Lemma~\ref{clm:T4upper}, we obtain
    \[e(G)\le
    f(0,0,t_3,t_5,m,i)+p(3t_4+2m+i+6,2m+i)\;.\]
    By~\eqref{defglt11}
    we have $e(G)\le g_\ell(0,0,t_3,t_4+2,m,i)$.
  \end{factproof}
 
  Observe that for
  any $n$-vertex, $(k+1)\times K_3$-free graph $G$ decomposed as in
  Setup~\ref{setup} we have $(t_1,t_2,t_3,t_4,m,i)\in F(n,k)$. By
  Claim~\ref{fac:glworks} there are $c_1,c_2\in\{0,1\}$ such that
  $(t_1-c_1,t_2-c_2,t_3,t_4+c_1+c_2,m,i)\in F(n,k)$ and such that $e(G)\le g_\ell(t_1-c_1,t_2-c_2,t_3,t_4+c_1+c_2,m,i)$.
  By Lemma~\ref{lem:maxgl} we thus obtain
  \[e(G)\le \max_{j\in[4]}e\big(E_j(n,k)\big)\;,\] as desired.
\end{proof}

\section{Graphs with few triangles touching a given set}\label{sec:fewtriangles}
In this section, we prove Lemma~\ref{lem:T4upperMG}. The extremal
problem of that lemma is not a very natural one. Also, we remark that Lemma~\ref{lem:T4upperMG} is sharp only
when $h\ge 9a$. This is the regime in
which we need the exact answer. 

However the closely related extremal problem of bounding the number of edges in a graph $H$ on $h$
vertices with no triangle touching a given set $A\subseteq V(H)$ of size
$a$ is quite natural. We already studied it in two previous
papers~\cite{AllBoettHlaPig_Turannical,AllBoettHlaPig_Strengthening}, where
we (respectively) determined the extremal function and proved uniqueness
and stability for the problem. We need a special case of the extremal
result of~\cite{AllBoettHlaPig_Turannical}.

\begin{theorem}
  \label{thm:TuranOrig}
  Let $H$ be a graph on $h$ vertices, and $A$ be a subset of $V(H)$ of size
  $a\le\tfrac{h}{2}$ such that no triangle of $H$ intersects $A$. Then we
  have
  \[e(H)\le \binom{h-2a}{2}+a(h-a)\,.\]
\end{theorem}

We also need a stability version of this theorem, proved
in~\cite{AllBoettHlaPig_Strengthening}. To this end, we consider the
following family~$\mathcal{H}_A$ of graphs on the vertex set $[h]$ and with
a distinguished set $A\subset[h]$, $|A|=a$ which show the optimality of
Theorem~\ref{thm:TuranOrig}.
To construct one graph in~$\mathcal{H}_A$, we take any set $B\subset[h]$ of size~$a$ disjoint from~$A$, put a complete balanced bipartite graph on $A\cup B$ (where the parts of the bipartite graph may be any partition of $A\cup B$) and make all the vertices of $[h]\setminus(A\cup B)$ adjacent to each other and to all the vertices of $B$. 
\begin{theorem}[\cite{AllBoettHlaPig_Strengthening}]
  \label{thm:TuranSt}
  For every $\eps>0$ there exist $\gamma>0$ and $h_0$ such that the following
  holds. Let $H$ be a graph of order $h\ge h_0$ and $A$ be a subset of $V(H)$ of
  size $a\le h/2$ such that no triangle of $H$ intersects $A$. Suppose furthermore that $e(H)\ge\binom{h-2a}{2}+a(h-a)-\gamma h^2$. Then by editing at most $\eps h^2$ pairs in $\binom{V(H)}{2}$ we can obtain a graph in~$\mathcal{H}_A$ (without changing the vertices of~$A$).
\end{theorem}

We will now show how this implies Lemma~\ref{lem:T4upperMG}.

\begin{proof}[Proof of Lemma~\ref{lem:T4upperMG}]
We set $\eps=1/400$ and let $\gamma>0$ and $h_0$ be given by Theorem~\ref{thm:TuranSt}. We set 
\begin{equation}\label{eq:choosekappa}
  \kappa_0=\max\big(10\gamma^{-1},h_0,8000\big)\,.
\end{equation}

Suppose that $h\ge\kappa_0$ and $H$ is an $h$-vertex graph and $A$ is a set
of~$a$ vertices such that no set of vertex disjoint triangles of $H$ covers
more than two vertices of $A$. This implies that we can identify a set of
at most two vertex-disjoint triangles covering a maximum number of vertices
of $A$. Taking the vertices of these triangles, and adding further arbitrary vertices if
necessary, we obtain a set $U$ of six vertices, with $|A\cap U|=2$, such
that $H-U$ has no triangle intersecting $A\setminus U$. Removing all the
edges of $H$ with one or two endpoints in $U$ therefore yields a graph $H'$
in which no triangle intersects~$A$. By Theorem~\ref{thm:TuranOrig}, $H'$
has at most $\binom{h-2a}{2}+a(h-a)$ edges.

There are two cases to deal with, corresponding to the two possibilities in
the definition of $p(h,a)$ in~\eqref{eq:defpprime}. The easier case is
$2a\le h< 9a$, where we do not attempt to prove a sharp extremal
result. Since at most $6h$ edges were removed from $H$ to obtain $H'$, we
have $e(H)\le\binom{h-2a}{2}+a(h-a)+6h=p(h,a)$, which completes the proof
in this case.

We now turn to the case $3\le a\le\tfrac{h}{9}$. Again, if $e(H')<
p(h,a)-6h$ then $e(H)<p(h,a)$ and we are done. So since $p(h,a)\ge
a(h-a)-2(h+2)+\binom{h-2a}{2}$ we may assume
\[e(H')>p(h,a)-6h\geByRef{eq:choosekappa}\binom{h-2a}{2}+a(h-a)-\gamma
h^2\,.\] By Theorem~\ref{thm:TuranSt} we can edit at most $\eps h^2$ pairs
in $\binom{V(H')}{2}$
we obtain an extremal graph $G\in\mathcal{H}_A$ on $V(H)$ with no triangle
intersecting $A$. Recall that $G$ consists of a complete balanced bipartite
graph on a set of $2a$ vertices $A\cup B$ (where~$A$ and~$B$ are not
necessarily the partition classes). The remaining vertices form a clique, and all the edges between them and~$B$ are present. It is
easy to check that since $|A|=|B|=a\le\frac{h}{9}$, any set of $\tfrac{2h}{9}$
vertices of $G$ induces at least
\[\binom{\tfrac{h}{9}}{2}\geByRef{eq:choosekappa} \tfrac{h^2}{200}\]
edges of $G$. Since $G$ was obtained from $H'$ by editing at most $\eps
h^2$ pairs in $\binom{V(H')}{2}$, and $H'$ was obtained from $H$ by deleting edges, it follows
that any set of $\tfrac{2h}{9}$ vertices of $H$ induce at least
$\big(\tfrac{1}{200}-\eps\big)h^2=\tfrac{1}{400}h^2$ edges of $H$.  

We claim that this implies that any set~$C$ of $\tfrac{2h}{9}$ vertices of
$H$ contains a matching with at least seven edges. Indeed, we can find such
a matching greedily, and after removing from~$C$ at most~$6$ matching edges
and all edges incident to these matching edges, we removed at most $12\cdot
\frac{2h}{9}+6<\frac1{400}h^2$ edges from~$C$.

Let $A\cap U=\{v_1,v_2\}$ and recall that $e(H')=e(H'- U)$.
Theorem~\ref{thm:TuranOrig} applied to the graph $H'-U$ on $h-6$ vertices
and the set $A\setminus U$ with $a-2$ vertices (it is indeed possible to apply Theorem~\ref{thm:TuranOrig} because $a-2\le\frac12(h-6)$ by~\eqref{eq:choosekappa} and by $a\le\tfrac{h}{9}$) gives
\[e(H')\le\binom{h-6-2(a-2)}{2}+(a-2)(h-6-(a-2))\,.\]
Observe that if $\deg_H(v_1)+\deg_H(v_2)\le 2h-6a-9$ then we have
\begin{align}
\nonumber
  e(H)&\le e(H') +\deg_H(v_1)+\deg_H(v_2)+4h\\
\nonumber  
  &\le \binom{h-6-2(a-2)}{2}+(a-2)\big(h-6-(a-2)\big) +2h-6a-9+4h\\
\label{eq:job}  
  &=p(h,a)\,,
\end{align}
and we are done. We may therefore assume
$\deg_H(v_1)+\deg_H(v_2)>2h-6a-9$, and since $\deg_H(v_i)\le h-1$ for
$i\in[2]$ it follows that $\deg_H(v_i)\ge h-6a-8\ge\tfrac{2h}{9}$, where
the final inequality follows from $a\le\tfrac{h}{9}$
and~\eqref{eq:choosekappa}.

Since any set of $\tfrac{2h}{9}$ vertices of $H$ contains a matching with
at least seven edges, we conclude that $N_H(v_i)$ contains such a
matching~$M_i$ for $i\in[2]$. Observe that if there were a triangle $xyz$
in $H$ with $x\in A\setminus\{v_1,v_2\}$ then we could use~$M_1$ and~$M_2$ 
to find greedily a collection of vertex-disjoint triangles in $H$
covering $\{x,v_1,v_2\}$. This is a contradiction to the assumption on $H$
that no such collection exists, and we conclude that there is no triangle
of $H$ which intersects $A':=A\setminus\{v_1,v_2\}$.

To complete the proof, we will now show that this final condition that no
triangle of $H$ intersects $A'$ implies that $e(H)\le p(h,a)$. Let~$\prec$
be a linear order of the vertices of~$H$. We apply the following `vertex
duplication' operation successively. If there are non-adjacent vertices
$u_1,u_2$ in $A'$ such that either $\deg_H(u_1)<\deg_H(u_2)$, or
$\deg_H(u_1)=\deg_H(u_2)$ and $u_1\prec u_2$, then we change $H$ by
resetting the neighbourhood of $u_1$ to $N_H(u_2)$. Let $H''$ be the graph
obtained by repeatedly applying this operation until every pair of
non-adjacent vertices of $A'$ has identical neighbourhoods.

By construction, we have $e(H'')\ge e(H)$, and no triangle in $H''$
intersects $A'$. Now $H''[A']$ is a complete partite graph, and since
$H''[A']$ contains no triangles it is a complete bipartite graph. Let its
parts be $Y_3$ and $Y_4$ (the latter of which may have size zero). 
Moreover, all vertices $y\in Y_3$ have the identical neighbourhood $N_{H''}(y)\setminus A'=:Y_1$. Likewise, all vertices $y\in Y_4$ have the identical neighbourhood $N_{H''}(y)\setminus A'=:Y_2$.
If $Y_4=\emptyset$ then we set
$Y_2=\emptyset$. Since no triangle of $H''$ intersects $A'$ the sets $Y_1$
and $Y_2$ are disjoint independent sets in $H''$. Finally, let $X$ be the
remaining vertices of $H''$. We have 
\begin{align*}
  e(H)\le e(H'')&\le \binom{|X|}{2}+\big(|Y_1|+|Y_2|\big)\big(|X|+|Y_3|+|Y_4|\big)\\
  &=\binom{h-a+2-s}{2}+s(h-s)\,, 
\end{align*} 
where $s:=|Y_1|+|Y_2|$. This function is maximised by $s=a-\tfrac{3}{2}$,
and the maximum with $s$ an integer occurs at $s=a-1,a-2$, where the
function is precisely equal to $p(h,a)$. We conclude that $e(H)\le p(h,a)$
as desired. \end{proof}

\section{Concluding remarks}\label{sec:conclusion}
\paragraph{\bf Small values of $n$}We did not try to optimise our arguments in order to reduce $n_0$. Indeed, the value we obtain depends on the relation between $\eps$ and $\gamma$ provided in Theorem~\ref{thm:TuranSt}, and the proof of that result in~\cite{AllBoettHlaPig_Strengthening} makes use of the Stability Theorem of Erd\H{o}s and Simonovits~\cite{E68,S68} for triangles. But there is no `heavy machinery' involved which would cause $n_0$ to become very large. It seems very likely that tracing exact values through these results would lead to a value of $n_0$ here smaller than $10^{10}$. Perhaps Theorem~\ref{thm:main} even holds with $n_0=1$, but we did not spend much effort on trying to find counterexamples for small values of $n$. Certainly our proof will not give such a result even with optimisation.

\medskip

\paragraph{\bf Tilings with larger cliques}It would be natural to ask for
an extension of Theorem~\ref{thm:main} to $(k+1)\times K_r$-free graphs $G$
rather than $(k+1)\times K_3$-free graphs, thus obtaining a density version
of the Hajnal--Szemer\'edi Theorem~\cite{hajnal70:_proof_p} rather than the Corr\'adi--Hajnal
Theorem. The same basic approach as in our proof of Theorem~\ref{thm:main}
seems to be a reasonable strategy for proving such a result: We call a
family $(\mathcal K_r, \mathcal K_{r-1}, \ldots, \mathcal K_1)$ an
\emph{$r$-tiling family} if $\mathcal K_i$ is a collection of disjoint
copies of the clique $K_i$ inside $G$, and the sets $\mathcal K_r, \mathcal
K_{r+1}, \ldots, \mathcal K_1$ partition the vertices of $G$. We then
consider an $r$-tiling family which maximises the vector $(|\mathcal K_r|,
|\mathcal K_{r+1}|, \ldots, |\mathcal K_1|)$ in lexicographic order,
and try to work out bounds on the edge counts inside the sets $\mathcal
K_i$ and between $\mathcal K_i$ and $\mathcal K_j$, relying again on
rotation techniques. Some parts of such an argument can be made to work, but
there are some additional difficulties for $r\ge 4$ that do not appear for
$r=3$. We are not even sure what the complete family of extremal graphs
should be.

\medskip

\paragraph{\bf Tilings with more general graphs}
An extension of Theorem~\ref{thm:main} which seems within the reach of existing
techniques is to get asymptotically tight bounds on the size of a maximal
$H$-tiling (as a function of the density of the host graph) for any
three-colourable graph $H$. The bipartite counterpart for this is the extension
of Theorem~\ref{thm:ErdGal} by Grosu and Hladk\'y~\cite{GroHla09+}. These problems can also be seen as density versions of Koml\'os's extension~\cite{komlosTT} of the Hajnal--Szemer\'edi Theorem to general graphs.
It seems likely that the technique developed by Koml\'os, and
adapted to this setting by Grosu and Hladk\'y, is flexible enough to allow such a generalisation
for $H$-tiling with any fixed 3-colourable graph $H$, and that the
extremal graphs for the problem of $H$-tiling in a graph of a given density will resemble the graphs $E_1,\ldots,E_4$ from
Definition~\ref{def:extremal}, though the part sizes will not be the same as in that definition.

\section{Acknowledgement}
The paper was finalised during the participation at the program Graphs,
Hypergraphs, and Computing at Institut Mittag--Leffler. We would like to
thank the organisers and the staff of the institute for creating a very
productive atmosphere. We had to take our little children with us. We would
like to acknowledge the support of the London Mathematical Society, and
Mathematics Institute at the University of Warwick (JH) for contributing to
childcare expenses that were incurred during this trip, and Emili Simonovits for helping us with babysitting.

Finally, we would like to thank an anonymous referee for their thorough comments.
\bibliographystyle{alpha}
\bibliography{bibl}

\newpage
\appendix 
\section{Maximisations}\label{sec:max}
In this section we provide proofs of
Lemmas~\ref{lem:maxf},
\ref{lem:maxgsmall}, and \ref{lem:maxgl}. These lemmas concern maximisations of certain functions.
We build our arguments on tedious elementary algebraic manipulations. While
some of the statements we need could be obtained by a more
routine technique of Lagrange multipliers, this method seems
to lead to even lengthier calculations in our setting. This is
caused in particular by a high degree of
discontinuity, caused by various case distinctions
and appearance of the floor/ceiling function, of the functions
we want to maximise.

We first collect some useful statements relating to $f$, all of which are
obtained by simple calculation using equations~\eqref{eq:deffprime}
and~\eqref{eq:deffixed}.
The three relations \eqref{eq:ft1t2}--\eqref{eq:upperf} below hold for any $\ttt_1,\ldots,\ttt_4,\mmm,\iii\ge 0$.
\begin{equation}\begin{split}\label{eq:ft1t2}
  f(\ttt_1+x,&\ttt_2-x,\ttt_3,\ttt_4,\mmm,\iii)-f(\ttt_1,\ttt_2,\ttt_3,\ttt_4,\mmm,\iii)\\
  &=\frac{x^2}{2}+\big(\mmm-\ttt_2-\ttt_3-\ttt_4+\tfrac{1}{2}\big)x-
  \begin{cases}0  & \mmm=0\\ 2x & \mmm>0\end{cases}
\end{split}\end{equation}

\begin{equation}\begin{split}\label{eq:ft3tot2}
  f(\ttt_1,\ttt_2+\ttt_3,0,\ttt_4,\mmm,\iii)-f(\ttt_1,\ttt_2,\ttt_3,\ttt_4,\mmm,\iii)\ge (\iii-3)\ttt_3
\end{split}\end{equation}

\begin{multline}\label{eq:upperf}
  f(\ttt_1,\ttt_2,\ttt_3,\ttt_4,\mmm,\iii)
  \le
  8\binom{\ttt_1+\ttt_2+\ttt_3+\ttt_4}{2}-8\binom{\ttt_4}{2} \\ +(4\mmm+2\iii+6)(\ttt_1+\ttt_2+\ttt_3)-\ttt_1\ttt_4
\end{multline}

Provided that $\min\{\mmm,\mmm-x,\iii+2x\}\ge 1$, $\iii\ge 0$, and $x\ge 0$ we have
\begin{equation}\begin{split}\label{eq:fmi}
  f(\ttt_1,\ttt_2,\ttt_3,\ttt_4,\mmm-x,\iii+2x)-f(\ttt_1,\ttt_2,\ttt_3,\ttt_4,\mmm,\iii)\ge
  x(\ttt_2-\ttt_3)\,.
\end{split}\end{equation}

If $\mmm\ge 5$, combining~\eqref{eq:ft3tot2} and~\eqref{eq:fmi} we have
\begin{multline}\label{eq:killt3}
  f(\ttt_1,\ttt_2+\ttt_3,0,\ttt_4,\mmm-4,\iii+8)-f(\ttt_1,\ttt_2,\ttt_3,\ttt_4,\mmm,\iii) \\
  \ge
  4(\ttt_2-\ttt_3)+(\iii+5)\ttt_3\ge 0\,.
\end{multline}

We will use the following lemma in our later maximisation results. Observe that Lemma~\ref{lem:maxf} is part~\ref{mff2a} of this lemma.

\begin{lemma}\label{lem:MaxFixFn}
  Given non-negative integers
  $\ttt_1,\ttt_2,\ttt_3,\ttt_4,\mmm,\iii$ the following are true.
  \begin{enumerate}[label=\rom]
    \item\label{mff1} We have
      \[f(\ttt_1,\ttt_2,\ttt_3,\ttt_4,\mmm,\iii)\le\max\big(f(\ttt_1+\ttt_2,0,\ttt_3,\ttt_4,\mmm,\iii),f(0,\ttt_1+\ttt_2,\ttt_3,\ttt_4,\mmm,\iii)\big)\]
      with equality only if either $\ttt_1=0$ or $\ttt_2=0$.
    \item\label{mff2} When $\iii\ge 4$ we have
      \[f(\ttt_1,\ttt_2,\ttt_3,\ttt_4,\mmm,\iii)\le f(\ttt_1,\ttt_2+\ttt_3,0,\ttt_4,\mmm,\iii)\]
      with equality only if $\ttt_3=0$.
    \item\label{mff2a} When
    $n\ge 3k+2$ we have
 \[\max_{(\ttt_1,\ttt_2,0,0,\mmm,\iii)\in
      F(n,k)}\big(f(\ttt_1,\ttt_2,0,0,\mmm,\iii)+\iii\mmm+\mmm^2\big)=\max_{j\in[3]}e\big(E_j(n,k)\big)\,.\]
    \item\label{mff3} When $n\ge 3k+21$ we have
      \[\max_{(\ttt_1,\ttt_2,\ttt_3,0,\mmm,\iii)\in
      F(n,k)}f(\ttt_1,\ttt_2,\ttt_3,0,\mmm,\iii)+\iii\mmm+\mmm^2=\max_{j\in[3]}e\big(E_j(n,k)\big)\,.\]
  \end{enumerate}
\end{lemma}
\begin{proof}[Proof of Lemma~\ref{lem:MaxFixFn}]
  {\sl Proof of part~\ref{mff1}:} By~\eqref{eq:ft1t2},
  \[f(\ttt_1+x,\ttt_2-x,\ttt_3,\ttt_4,\mmm,\iii)-f(\ttt_1,\ttt_2,\ttt_3,\ttt_4,\mmm,\iii)\] is a
  quadratic in $x$ with positive $x^2$-coefficient. It follows
  that for any $a\le b$, the maximum of
  $f(\ttt_1+x,\ttt_2-x,\ttt_3,\ttt_4,\mmm,\iii)-f(\ttt_1,\ttt_2,\ttt_3,\ttt_4,\mmm,\iii)$ over
  $a\le x\le b$ occurs when either $x=a$ or $x=b$. In
  particular, we have for all non-negative
  $\ttt_1,\ttt_2,\ttt_3,\ttt_4,\mmm,\iii$ that
  \[f(\ttt_1,\ttt_2,\ttt_3,\ttt_4,\mmm,\iii)\le\max\big(f(\ttt_1+\ttt_2,0,\ttt_3,\ttt_4,\mmm,\iii),f(0,\ttt_1+\ttt_2,\ttt_3,\ttt_4,\mmm,\iii)\big)\;,\]
  with equality only when $\ttt_1=0$ or $\ttt_2=0$.
  
  {\sl Proof of part~\ref{mff2}:} By~\eqref{eq:ft3tot2}, when
  $\iii\ge 4$, we have \[f(\ttt_1,\ttt_2+\ttt_3,0,\ttt_4,\mmm,\iii)\ge
  f(\ttt_1,\ttt_2,\ttt_3,\ttt_4,\mmm,\iii)\;,\] with equality only when $\ttt_3=0$.
  
  {\sl Proof of part~\ref{mff2a}:} By part~\ref{mff1} the maximum on the
  left-hand side is attained either when $\ttt_1=k,\ttt_2=0$, or when
  $\ttt_1=0,\ttt_2=k$.
  
  By~\eqref{eq:deffprime} and~\eqref{eq:deffixed} we have
  \begin{equation}\begin{split}
    f(k,0,0,0,\mmm,\iii)+\iii\mmm+\mmm^2 & =4\mmm k+2\iii k+7\binom{k}{2}+3k+\iii\mmm+\mmm^2 \\
    & =7\binom{k}{2}+3k+2(n-3k)k+\mmm(n-3k-\mmm) \\
    & \le 7\binom{k}{2}+3k+2(n-3k)k+
    \Big\lfloor\frac{n-3k}{2}\Big\rfloor\Big\lceil\frac{n-3k}{2}\Big\rceil \\
    & = \binom{k}{2}+k(n-k)+
    \Big\lfloor\frac{n-k}{2}\Big\rfloor\Big\lceil\frac{n-k}{2}\Big\rceil\\
    &=e\big(E_1(n,k)\big)\;,
  \end{split} \label{eq:MAXIMA11}
  \end{equation}
  where the last term on the second line achieves its maximum,
  $\big\lfloor\frac{n-3k}{2}\big\rfloor\big\lceil\frac{n-3k}{2}\big\rceil$,
  exactly when $\mmm=n-3k-\mmm$ and $\iii=0$, if $n-3k$ is even, or when $\mmm=n-3k-\mmm-1$
  and $\iii=1$, if not (observe that we cannot have $\mmm=n-3k-\mmm+1$, since then we
  would have $\iii=-1$).
  
  To deal with the term $\ttt_1=0$, $\ttt_2=k$ we have to distinguish between
  the cases $\mmm=0$ and $\mmm>0$.
  
   When $\mmm=0$ we first observe that the case $k=0$ trivially satisfies the
   statement.
   Thus we assume that $k>0$. We have
  \begin{equation}
\label{eq:MAXIMA12}  
  \begin{split}
    f(0,k,0,0,0,n-3k)&+0+0 \\  
    &=2(n-3k)k+8\binom{k}{2}+3k\\
    & < 8\binom{k}{2}+2(n-3k-2)k+3k+5k\\
    & = f(0,k,0,0,1,n-3k-2)\\
    &\le f(0,k,0,0,1,n-3k-2)+(n-3k-2)\times 1+1^2\,.
  \end{split}\end{equation}
  It follows that
  $f(0,k,0,0,\mmm,\iii)+\iii\mmm+\mmm^2$ is not maximised on $F(n,k)$ when $\mmm=0$. 
  
 When $\mmm\ge 1$, again from~\eqref{eq:deffprime} and~\eqref{eq:deffixed}, we have
  \begin{equation}
\label{eq:MAXIMA13}  
  \begin{split}
    f(0,k,0,0,\mmm,\iii)+\iii\mmm+\mmm^2 & =2\iii k+8\binom{k}{2}+3k+(2+3\mmm)k+\iii\mmm+\mmm^2\\
    & = 8\binom{k}{2}+5k+2(n-3k)k+\mmm(n-\mmm-4k)\,,
  \end{split}\end{equation}
  which is maximised on $F(n,k)$ both when
  \[\mmm=\max\Big(1,\Big\lfloor\frac{n-4k}{2}\Big\rfloor\Big)\mbox{ and } 
  \mmm=\max\Big(1,\Big\lceil\frac{n-4k}{2}\Big\rceil\Big)\,.\]
  
 It is
  straightforward from~\eqref{eq:deffprime} and~\eqref{eq:deffixed} to check that for the numbers 
  $\mmm_1:=\lfloor\frac{n-4k}{2}\rfloor, \iii_1:=n-3k-2\mmm_1$, and $\mmm_2:=\lceil\frac{n-4k}{2}\rceil, \iii_2:=n-3k-2\mmm_2$ we have 
  \begin{equation*}\begin{split}
  f(0,k,0,0,\mmm_1,
  \iii_1) +\iii_1\mmm_1+\mmm_1^2&= f(0,k,0,0,\mmm_2,
  \iii_2) +\iii_2\mmm_2+\mmm_2^2\\
  &=e\big(E_2(n,k)\big)\;.
  \end{split}\end{equation*}
  Further
  \begin{equation}
\label{eq:MAXIMA14}  
  f(0,k,0,0,1,n-3k-2)+
  (n-3k-2)\times 1+1^2=e\big(E_3(n,k)\big)\,.
  \end{equation}
  This
  completes the proof.

  {\sl Proof of part~\ref{mff3}:} Let $k\ge 0$ and $n\ge 3k+21$ be fixed.
  By part~\ref{mff2a} it is enough to show that the function
  $f(\ttt_1,\ttt_2,\ttt_3,0,\mmm,\iii)+\iii\mmm+\mmm^2$ is maximised on the set
  $F(n,k)$ only when $\ttt_3=0$.
  
  Let $(\ttt_1,\ttt_2,\ttt_3,0,\mmm,\iii)\in F(n,k)$. From part~\ref{mff2} we have that if
  $\iii\ge 4$ then
  \begin{multline*}
    f(\ttt_1,\ttt_2,\ttt_3,0,\mmm,\iii) \\
    \le\max\big(f(\ttt_1+\ttt_2+\ttt_3,0,0,0,\mmm,\iii),f(0,\ttt_1+\ttt_2+\ttt_3,0,0,\mmm,\iii)\big)\,.
  \end{multline*}
  with equality only when $\ttt_3=0$, as desired. In the
  rest of the proof we assume that $\iii\le 3$. Since
  $2\mmm+\iii=n-3k\ge 21$, we have $\mmm\ge 9$.
  
  We separate two cases.
  First, suppose that $\ttt_2+\ttt_3\ge 13$. Using~\eqref{eq:fmi}, since $\mmm-6\ge 3\ge
  1$, we have
  \begin{align*}f(\ttt_1,\ttt_2+\ttt_3,0,0,\mmm-6,\iii+12)&+(\iii+12)(\mmm-6)+(\mmm-6)^2\\
  &-\left(f(\ttt_1,\ttt_2+\ttt_3,0,0,\mmm,\iii)+\iii\mmm+\mmm^2\right)\\
  &\ge 6(\ttt_2+\ttt_3)-6\iii-36>(3-\iii)\ttt_3\;,\end{align*}
  since we have $\ttt_2+\ttt_3\ge 13$. By~\eqref{eq:ft3tot2} we obtain 
  \begin{align*}
  f(\ttt_1,\ttt_2+\ttt_3,0,0,\mmm-6,\iii+12)&+(\iii+12)(\mmm-6)+(\mmm-6)^2\\
  &>f(\ttt_1,\ttt_2,\ttt_3,0,\mmm,\iii)+\iii\mmm+\mmm^2\,,
  \end{align*} as desired.
  
  Second, suppose that $\ttt_2+\ttt_3\le 12$. We have 
  \begin{equation*}\begin{split}
    f(\ttt_1+\ttt_2+\ttt_3,0,0,0,\mmm,\iii) &-f(\ttt_1,\ttt_2+\ttt_3,0,0,\mmm,\iii)\\
    &\ge (\mmm-2)(\ttt_2+\ttt_3)-\binom{\ttt_2+\ttt_3}{2} \\
    &>(\mmm+\iii-11)(\ttt_2+\ttt_3)+(3-\iii)\ttt_3\\
    &\ge (3-\iii)\ttt_3\,,  
  \end{split}\end{equation*}
  where the last inequality comes from $2\mmm+\iii=n-3k\ge 21$.
  Together with~\eqref{eq:ft3tot2} we then have
  \[f(\ttt_1+\ttt_2+\ttt_3,0,0,0,\mmm,\iii)-f(\ttt_1,\ttt_2,\ttt_3,0,\mmm,\iii)>0\;,\] and
  hence that
  \[f(\ttt_1+\ttt_2+\ttt_3,0,0,0,\mmm,\iii)+\iii\mmm+\mmm^2>f(\ttt_1,\ttt_2,\ttt_3,0,\mmm,\iii)+\iii\mmm+\mmm^2\;,\]
  as desired.
  \end{proof}

\begin{proof}[Proof of Lemma~\ref{lem:maxgsmall}]
  Let $\bar \ttt_1:=\ttt_2+\ttt_3+\ttt_4$. We now provide some preliminary
  observations and then distinguish four cases to prove the lemma.

  From~\eqref{eq:Defngsmall} we have
  \begin{equation}\label{eq:NiceT4:eG-f}
      g_s(\ttt_1,\ttt_2,\ttt_3,\ttt_4,\mmm,\iii)-g_s(\ttt_1,\bar \ttt_1,0,0,\mmm,\iii) 
 \le 9\ttt_4-(\ttt_3+\ttt_4)(\iii-3)\,.
\end{equation}
  Moreover, if $k\le 43n/140$ then we have $2\mmm+\iii=n-3k\geq\frac{11n}{140}$. Hence
  \begin{equation}\label{eq:NiceT4:m}
    \mmm\geq\frac{11n}{280}-6
    \qquad \text{if $k\le 43n/140$ and $\iii\le 12$} \,.
\end{equation}

  \smallskip

  {\sl Case~1:} $k>43n/140$.
  We shall show that this implies \[g_s(\ttt_1,\ttt_2,\ttt_3,\ttt_4,\mmm,\iii)<
  e\big(E_4(n,k)\big)\,.\] From~\eqref{eq:upperf} we have
  \begin{equation}
  \label{eq:MAXIMA15}
    \begin{split}
    & g_s(\ttt_1,\ttt_2,\ttt_3,\ttt_4,\mmm,\iii)\\
    \le &
    8\binom{\ttt_1+\ttt_2+\ttt_3+\ttt_4}{2}+(4\mmm+2\iii+15)(\ttt_1+\ttt_2+\ttt_3+\ttt_4)-28+\iii\mmm+\mmm^2\\
    \le & 8\binom{k}{2}+2k(n-3k)+20k+\Big(\frac{n-3k}{2}\Big)^2\,,
    \end{split}
  \end{equation}
  where the second inequality comes from the fact that $(\ttt_1,\ttt_2,\ttt_3,\ttt_4,\mmm,\iii)\in F(n,k)$.
  
  Now solving the quadratic inequality (in variable
  $k$)
  \begin{multline*}
     8\binom{k}{2}+20k+2(n-3k)k+\Big(\frac{n-3k}{2}\Big)^2 \\
      <
      \binom{6k-n+4}{2}+(6k-n+4)(n-3k-2)+(n-3k-2)^2 \\
      = e\big(E_4(n,k)\big) 
 \end{multline*}
  shows that we have $g_s(\ttt_1,\ttt_2,\ttt_3,\ttt_4,\mmm,\iii)< e\big(E_4(n,k)\big)$ if~$k$ satisfies
\begin{equation}\label{eq:NiceT4:k}
 k>\frac{n+2}{5}+\frac{\sqrt{14n^2+406n-84}}{35}
 =\frac{7+\sqrt{14}}{35}n+\frac{14+15\sqrt{14}}{35}
  \,.
\end{equation}
 Indeed,~\eqref{eq:NiceT4:k} is satisfied as $n\ge 8406$ and as $k>\frac{43n}{140}$. Hence we are done in this case.

  \smallskip

  {\sl Case~2:} $\iii\ge 12$. Note that the right hand side
  of~\eqref{eq:NiceT4:eG-f} is not positive for $\iii\ge 12$.
  Thus~\eqref{eq:NiceT4:eG-f} implies $g_s(\ttt_1,\ttt_2,\ttt_3,\ttt_4,\mmm,\iii)\le g_s(\ttt_1,\bar
  \ttt_1,0,0,\mmm,\iii)=f(\ttt_1,\bar \ttt_1,0,0,\mmm,\iii)+\iii\mmm+\mmm^2-28$. Moreover,
  by Lemma~\ref{lem:MaxFixFn} part~\ref{mff2a} the function $f(\ttt_1,\bar
  \ttt_1,0,0,\mmm,\iii)+\iii\mmm+\mmm^2$, subject to the constraints $\ttt_1+\bar \ttt_1=k$ and
  $2\mmm+\iii=n-3k$, is bounded from above by $\max_{j\in[3]}
   e\big(E_j(n,k)\big)$. 
   Hence
  $g_s(\ttt_1,\ttt_2,\ttt_3,\ttt_4,\mmm,\iii)\le \max_{j\in[3]}
  e\big(E_j(n,k)\big)$ as desired.

  \smallskip

  {\sl Case~3:} $k\le 43n/140$, $\iii<12$, $\bar \ttt_1>80$.
  By~\eqref{eq:NiceT4:m} we have $\mmm-25>0$, so from~\eqref{eq:Defngsmall} we
  obtain
  \begin{multline}
  \label{eq:MAXIMA16}
    g_s(\ttt_1,\bar \ttt_1,0,0,\mmm-25,\iii+50)
    -g_s(\ttt_1,\bar \ttt_1,0,0,\mmm,\iii)
    \ge 25 \bar \ttt_1 - 25\iii - 25^2 \\
    \gBy{$\iii<12$} 25 \bar \ttt_1-25\cdot 37
    \gBy{$\bar \ttt_1>80$} 12 \bar \ttt_1 + 13\cdot 80 - 25\cdot 37
    > 12\bar \ttt_1 \ge 12 (\ttt_3+\ttt_4)
    \,.
  \end{multline}
  Since the right hand side of~\eqref{eq:NiceT4:eG-f} is at most
  $12(\ttt_3+\ttt_4)$, this implies
  \begin{equation*}\begin{split}
      g_s(\ttt_1,\bar \ttt_1,0,0,\mmm-25,\iii+50)
      &> g_s(\ttt_1,\bar \ttt_1,0,0,\mmm,\iii)+12(\ttt_3+\ttt_4)\\
      &\ge g_s(\ttt_1,\ttt_2,\ttt_3,\ttt_4,\mmm,\iii)\,,
  \end{split}\end{equation*}
  and so we conclude from Lemma~\ref{lem:MaxFixFn} part~\ref{mff2a} that
  \begin{equation*}\begin{split}
    g_s(\ttt_1,\ttt_2,\ttt_3,\ttt_4,\mmm,\iii)& <g_s(\ttt_1,\bar \ttt_1,0,0,\mmm-25,\iii+50)\\
    & <f(\ttt_1,\bar
    \ttt_1,0,0,\mmm-25,\iii+50)\\&~~~+(\iii+50)(\mmm-25)+(\mmm-25)^2\\&\le
    \max_{j\in[3]} e\big(E_j(n,k)\big)\,.
  \end{split}\end{equation*}

  \smallskip

  {\sl Case~4:} $k\le 43n/140$, $\iii<12$, $\bar \ttt_1\le 80$.
  Again from~\eqref{eq:Defngsmall} we have 
  \begin{equation}
\label{eq:MAXIMA17}  
  \begin{split}
      g_s(\ttt_1 +\bar \ttt_1,0,0,0,\mmm,\iii)-g_s(\ttt_1,\bar \ttt_1,0,0,\mmm,\iii)
      \ge \bar \ttt_1\mmm-\binom{\bar \ttt_1}{2}-2\bar \ttt_1\,.
  \end{split}\end{equation}
  In addition, by~\eqref{eq:NiceT4:m} and since $n\ge 8406$ we have
  $\mmm\geq\frac{11n}{280}-6>320$. This implies 
  \begin{equation*}
    \bar \ttt_1 \mmm-2\bar \ttt_1-\binom{\bar \ttt_1}{2}
    = \bar \ttt_1 \Big( \mmm-\frac{\bar \ttt_1-1}{2}-2 \Big)
    >\bar \ttt_1 \cdot 12 
    \ge 12(\ttt_3+\ttt_4) \,,
  \end{equation*}
  and hence we obtain using Lemma~\ref{lem:MaxFixFn} part~\ref{mff2a} that
  \begin{align*}
  \max_{j\in[3]}
  e\big(E_j(n,k)\big)&\ge
  f(\tau_1+\bar\ttt_1,0,0,0,\mmm,\iii)+\iii\mmm+\mmm^2
  \\ &>g_s(\ttt_1 +\bar
  \ttt_1,0,0,0,\mmm,\iii)\\ &> g_s(\ttt_1,\bar
  \ttt_1,0,0,\mmm,\iii)+12(\ttt_3+\ttt_4)\\ 
  &\ge g_s(\ttt_1,\ttt_2,\ttt_3,\ttt_4,\mmm,\iii)\;,
\end{align*}
  where, again, the last inequality follows from~\eqref{eq:NiceT4:eG-f}.
\end{proof}

\begin{proof}[Proof of Lemma~\ref{lem:maxgl}]
  Our aim is to show that
  $g_\ell(\ttt_1,\ttt_2,\ttt_3,\ttt_4,\mmm,\iii)$ with all variables
  required to be non-negative integers and with
  $\ttt_1+\ttt_2+\ttt_3+\ttt_4=k$ and $2\mmm+\iii=n-3k$, is bounded above
  by
  \[e(n,k):=\max_{i\in[4]}\big\{e\big(E_1(n,k)\big),\ldots,e\big(E_4(n,k)\big)\big\}\,.\]
  The
  main difficulty is to show that indeed if $g_\ell$ is maximised then
  $\ttt_3=0$ and at most one of the variables $\ttt_1,\ttt_2,\ttt_4$ is
  non-zero, and as mentioned the reason why this seems not to be easy to
  automate is that $g_\ell$ is quite discontinuous. There are two regimes
  in which $g_\ell$ behaves quite differently. Furthermore, it is
  occasionally convenient to assume that $2\mmm+\iii$ is reasonably large,
  leading to a third case.

The easier of the two regimes is when $3\ttt_4<2\mmm+\iii$. In this case, $g_\ell$ is defined by~\eqref{defglsmt4}. This function is still not quite continuous: when $\mmm$ or $\iii$ are changed from $0$ to $1$ or vice versa, there is discontinuity in the definition of the function $f$, but this turns out to be easy to handle.

{\sl Case 1:} $3\ttt_4<n-3k$ and $n-3k\ge 30$.

We define the following auxiliary function.
\begin{align}
\begin{split}\label{eq:defhLM}
h(\ttt_1,\ttt_2,\ttt_3,\ttt_4,\mmm,\iii):=4\mmm\ttt_1+2\iii\ttt_1+7\binom{\ttt_1}{2}+3\ttt_1+2\iii\ttt_2+8\binom{\ttt_2}{2}+3\ttt_2+\\
8\binom{\ttt_3}{2}+8\ttt_3\ttt_4+3\ttt_3+7\ttt_1\ttt_2+(2+3\mmm)\ttt_2+7\ttt_1(\ttt_3+\ttt_4)+
(3+3\mmm)\ttt_3+\\
8\ttt_2(\ttt_3+\ttt_4)+(2+\iii)\ttt_3+\iii\mmm+\mmm^2+(3+3\mmm)\ttt_4+(2+\iii)\ttt_4+\binom{3\ttt_4}{2}\,.
\end{split}
\end{align}
Observe that this function is almost the same as $g_\ell$: indeed, if $\mmm,\iii\ge 1$ then they are equal, while otherwise $g_\ell$ is smaller and the difference is one of $2\ttt_2+3\ttt_3$, $2\ttt_3$ and $2\ttt_2+5\ttt_3$ according to~\eqref{eq:deffixed}. We have the following equations (where we write $h$ for $h(\ttt_1,\ttt_2,\ttt_3,\ttt_4,\mmm,\iii)$).
\begin{align}
\label{eq:E1ax}
h(\ttt_1+x,\ttt_2,\ttt_3,\ttt_4-x,\mmm,\iii)-h&=x^2+\mmm x+\iii x-4x-\ttt_3x-\ttt_2x-2\ttt_4x\\
\label{eq:E1bx}
h(\ttt_1,\ttt_2+x,\ttt_3,\ttt_4-x,\mmm,\iii)-h&=\tfrac{x^2}{2}+\iii x-\tfrac52 x-\ttt_4x\\
\label{eq:E1cx}
h(\ttt_1,\ttt_2,\ttt_3+x,\ttt_4-x,\mmm,\iii)-h&=\tfrac{x^2}{2}+\tfrac{x}{2}-t_4x
\end{align}
These are all quadratic in $x$ with positive $x^2$ coefficients, and by the above observation the same statement is true (though the linear terms are different) when $h$ is replaced with $g_\ell$ throughout. It follows that if $g_\ell$ is maximised, then either $\ttt_4=0$ or $\ttt_4=\tfrac{n-3k}{3}$. We have the following equations:
\begin{align}
\label{ht1t3sm}h(\ttt_1+x,\ttt_2,\ttt_3-x,\ttt_4,\mmm,\iii)-h&=\tfrac{x^2}{2}+(\mmm+\iii-\ttt_2-\ttt_3-\ttt_4-\tfrac92)x\\
\label{ht1t2sm}h(\ttt_1+x,\ttt_2-x,\ttt_3,\ttt_4,\mmm,\iii)-h&=\tfrac{x^2}{2}+(\mmm-\ttt_2-\ttt_3-\ttt_4-\tfrac32)x\\
\label{ht2t3sm}h(\ttt_1,\ttt_2+x,\ttt_3-x,\ttt_4,\mmm,\iii)-h&=(\iii-3)x
\end{align}
First we will consider the case $\ttt_4=\tfrac{n-3k}{3}>0$. The equations~\eqref{ht1t3sm} and~\eqref{ht1t2sm} are positive quadratics in $x$, and the same is true replacing $h$ with $g_\ell$ throughout. It follows that if $g_\ell$ is maximised and $\ttt_1>0$ then $\ttt_1=2k-\tfrac{n}{3}$ and $\ttt_2=\ttt_3=0$. In this case we have $h=g_\ell$, and also
\begin{align}
\begin{split}\label{eq:ceLM}
g_\ell(k,0,0,0,\mmm,\iii)-g_\ell(2k-\tfrac{n}{3},0,0,\tfrac{n-3k}{3},\mmm,\iii)&=\tfrac{n-3k}{3}\big(-\tfrac{n-3k}{3}+\mmm +\iii -4\big)\\
&= \tfrac{n-3k}{3}\big(\tfrac{n-3k}{6}+\tfrac{\iii}{2}-4\big)\,,
\end{split}
\end{align}
which, since $n-3k\ge 30$, is positive. This is a contradiction to $g_\ell$ being maximised.

It remains to check the case $\ttt_1=0$. By~\eqref{ht2t3sm}, we have $h\le h(\ttt_1,\ttt_2+\ttt_3,0,\ttt_4,\mmm,\iii)+3\ttt_3$, and so the total difference between $h(\ttt_1,\ttt_2+\ttt_3,\ttt_4,\mmm,\iii)$ and $g_\ell(\ttt_1,\ttt_2,\ttt_3,\ttt_4,\mmm,\iii)$ is at most $8k\le 3n$. Since we assume $\ttt_1=0$ and $\ttt_4=\tfrac{n-3k}{3}$, and since $n\ge 10^4$, it is enough to show that $h(0,2k-\tfrac{n}{3},0,\tfrac{n-3k}{3},\mmm,\iii)$ is always smaller than $e(n,k)$ by at least $\tfrac{1}{1000}n^2$. To simplify the analysis, we write $\approx$ to mean we discard all terms only linear in $n$. Together with the difference between $h$ and $g_\ell$, the linear error terms never amount to more than $6n<\tfrac{1}{1000}n^2$.

We have
\begin{equation}\label{eq:MAXIMA1}
h(0,2k-\tfrac{n}{3},0,\tfrac{n-3k}{3},\mmm,\iii)=3\iii k-\tfrac{\iii n}{3}+\tfrac{n^2}{18}-\tfrac{3k}{2}+\tfrac{5n}{6}+3k\mmm-\tfrac{kn}{3}+\tfrac{9k^2}{2}+\iii\mmm+\mmm^2\,.
\end{equation}
Discarding the linear terms and substituting $\iii=n-3k-2m$ we get
\begin{equation}\label{eq:MAXIMA2}
h(0,2k-\tfrac{n}{3},0,\tfrac{n-3k}{3},\mmm,\iii)\approx \tfrac{11kn}{3}-\tfrac{9k^2}{2}-6k\mmm-\tfrac{5n^2}{18}+\tfrac{5\mmm n}{3}-\mmm^2\;.
\end{equation}

This function is a negative quadratic in $\mmm$ with maximum at $\mmm=\tfrac{5n}{6}-3k$. Since we are only interested in the case that $\mmm\ge 0$ and $\iii\ge 0$, we need to separate some subcases.

Subcase 1: $0.31n\le k\le n/3$ and $\mmm=0$. We have
\begin{align}
\label{eq:MAXIMA3}
h(0,2k-\tfrac{n}{3},0,\tfrac{n-3k}{3},0,n-3k)&\approx\tfrac{11kn}{3}-\tfrac{9k^2}{2}-\tfrac{5n^2}{18}\;\text{and}\\
\label{eq:MAXIMA4}
h(0,2k-\tfrac{n}{3},0,\tfrac{n-3k}{3},0,n-3k)-e(E_4(n,k))&\approx-\tfrac{7n^2}{9}+\tfrac{20kn}{3}-\tfrac{27k^2}{2}\,.
\end{align}
This function is maximised at $k=\tfrac{20n}{81}$, and since $0.31>\tfrac{20}{81}$ its maximum in the range $0.31n\le k\le n$ is at $k=0.31n$, where the value attained is less than $-0.007n^2$.

Subcase 2: $\tfrac{5n}{18}\le k<0.31n$ and $\mmm=0$. We have
\begin{equation}\label{eq:MAXIMA5}
h(0,2k-\tfrac{n}{3},0,\tfrac{n-3k}{3},0,n-3k)-e(E_3(n,k))\approx\tfrac{5kn}{3}-\tfrac{5k^2}{2}-\tfrac{5n^2}{18}\,,\end{equation}
which function is maximised at $k=n/3$ and hence in the range of $k$ always smaller than the value at $k=0.31n$, which is less than $-0.001n^2$.

Subcase 3: $\tfrac{2n}{9}\le k\le\tfrac{5n}{18}$, and $\mmm=\tfrac{5n}{6}-3k$. We have
\begin{equation}
\label{eq:MAXIMA6}
h(0,2k-\tfrac{n}{3},0,\tfrac{n-3k}{3},\tfrac{5n}{6}-3k,3k-\tfrac{2n}{3})\approx-\tfrac{4kn}{3}+\tfrac{9k^2}{2}+\tfrac{15n^2}{36},\quad\text{and}
\end{equation}
\begin{equation}
\label{eq:MAXIMA7}
 h(0,2k-\tfrac{n}{3},0,\tfrac{n-3k}{3},\tfrac{5n}{6}-3k,3k-\tfrac{2n}{3})-e(E_3(n,k))\approx-\tfrac{10kn}{3}+\tfrac{13k^2}{2}+\tfrac{15n^2}{36}\,,
\end{equation}
which is a positive quadratic in $k$. At $k=\tfrac{5n}{18}$ the value of the LHS of~\eqref{eq:MAXIMA7} is $\tfrac{-1205n^2}{648}$, and at $k=\tfrac{2n}{9}$ we get $\tfrac{-481n^2}{324}$, the latter of which is the maximum in this range of $k$.

Subcase 4: $\tfrac{n}{5}<k<\tfrac{2n}{9}$ and $\iii=0$. We get
\begin{equation}
\label{eq:MAXIMA8}
h(0,2k-\tfrac{n}{3},0,\tfrac{n-3k}{3},\tfrac{n-3k}{2},0)\approx\tfrac{-kn}{3}+\tfrac{11n^2}{36}+\tfrac{9k^2}{4},\quad \text{and}
\end{equation}
\begin{equation}
\label{eq:MAXIMA9}
\quad h(0,2k-\tfrac{n}{3},0,\tfrac{n-3k}{3},\tfrac{n-3k}{2},0)-e(E_1(n,k))\approx\tfrac{5k^2}{2}-\tfrac{5kn}{6}+\tfrac{n^2}{18}\,,
\end{equation}
which is a positive quadratic in $k$. At $k=\tfrac{n}{5}$ we get $\tfrac{-n^2}{90}$, and at $k=\tfrac{2n}{9}$ the value is $\tfrac{-n^2}{162}$, the latter of which is the maximum in this range of $k$.

Since these subcases exhaust the range of $k$ we are considering, we conclude that indeed if $\ttt_1=0$ and $\ttt_4>0$ then $g_\ell$ is not maximised. It follows that the only maxima of $g_\ell$ with $3\ttt_4<2\mmm+\iii$ are with $\ttt_4=0$. By Lemma~\ref{lem:MaxFixFn} part~\ref{mff3}, it follows that the maximum of $g_\ell(\ttt_1,\ttt_2,\ttt_3,\ttt_4,\mmm,\iii)$ subject to the conditions $\ttt_1+\ttt_2+\ttt_3+\ttt_4=k$, $2\mmm+\iii=n-3k$, and $3\ttt_4<2\mmm+\iii$, is at most $e(n,k)$ as desired.

{\sl Case 2:} $3\ttt_4\ge \max(528,3\kappa_0,n-3k)$.

In this range $g_\ell$ is defined by either~\eqref{defglnorm} or~\eqref{defglt11}. As in the previous case, the function is not continuous but the discontinuities are small. Again we define an auxiliary function:
\begin{align}
\begin{split}\label{eq:hDVA}
h(\ttt_1,\ttt_2,\ttt_3,\ttt_4,\mmm,\iii):=4\mmm\ttt_1+2\iii\ttt_1+7\binom{\ttt_1}{2}+3\ttt_1+2\iii\ttt_2+8\binom{\ttt_2}{2}+3\ttt_2+\\
8\binom{\ttt_3}{2}+8\ttt_3\ttt_4+3\ttt_3+7\ttt_1\ttt_2+(2+3\mmm)\ttt_2+7\ttt_1(\ttt_3+\ttt_4)+
(3+3\mmm)\ttt_3+\\
8\ttt_2(\ttt_3+\ttt_4)+(2+\iii)\ttt_3+(2\mmm+\iii-2)(3\ttt_4+2)+\binom{3\ttt_4-2\mmm-\iii+4}{2}\,.
\end{split}
\end{align}

The difference between $h$ and $g_\ell$ is at most $2\ttt_2+5\ttt_3+12n+20$. As in the previous case, we compute some differences of this auxiliary function, where we write $h$ as a shorthand for $h(\ttt_1,\ttt_2,\ttt_3,\ttt_4,\mmm,\iii)$.
\begin{align}
\label{ht1t2l}h(\ttt_1+x,\ttt_2-x,\ttt_3,\ttt_4,\mmm,\iii)-h&=\tfrac{x^2}{2}+(\mmm-\ttt_2-\ttt_3-\ttt_4-\tfrac32)x\\
\label{ht1t3l}h(\ttt_1+x,\ttt_2,\ttt_3-x,\ttt_4,\mmm,\iii)-h&=\tfrac{x^2}{2}+(\mmm+\iii-\ttt_2-\ttt_3-\ttt_4-\tfrac92)x\\
\label{ht2t3l}h(\ttt_1,\ttt_2+x,\ttt_3-x,\ttt_4,\mmm,\iii)-h&=(\iii-3)x\\
\label{ht1t4l}h(\ttt_1+x,\ttt_2,\ttt_3,\ttt_4-x,\mmm,\iii)-h&=x^2+(4\mmm+2\iii-\ttt_2-\ttt_3-2\ttt_4-5)x\\
\label{ht2t4l}h(\ttt_1,\ttt_2+x,\ttt_3,\ttt_4-x,\mmm,\iii)-h&=\tfrac{x^2}{2}+(2\iii+3\mmm-\ttt_4-\tfrac72)x\\
\label{ht3t4l}h(\ttt_1,\ttt_2,\ttt_3+x,\ttt_4-x,\mmm,\iii)-h&=\tfrac{x^2}{2}+(3\mmm-\ttt_4+\tfrac12)x
\end{align}

We will first consider the case $\ttt_4=\tfrac{n-3k}{3}$. We will show that in this case $g_\ell<e(n,k)$. By~\eqref{ht2t3l}, $h$ is larger by at most $3t_3$ than $h(\ttt_1,\ttt_2+\ttt_3,0,\ttt_4,\mmm,\iii)$, and so $g_\ell(\ttt_1,\ttt_2,\ttt_3,\ttt_4,\mmm,\iii)$ is larger than $h(\ttt_1,\ttt_2+\ttt_3,0,\ttt_4,\mmm,\iii)$ by at most $2\ttt_2+8\ttt_3+12n+20<16n$. Since $n\ge 4\cdot 10^4$ it suffices to show that $h(\ttt_1,\ttt_2,0,\tfrac{n-3k}{3},\mmm,\iii)$ is smaller than $e(n,k)$ by at least $\tfrac{1}{2000}n^2$.

We may assume that $h(\ttt_1,\ttt_2,0,\tfrac{n-3k}{3},\mmm,\iii)$ is maximised, and since~\eqref{ht1t2l} is a positive quadratic in $x$ this implies that either $\ttt_1=2k-\tfrac{n}{3}$ and $\ttt_2=0$ or vice versa; we separate subcases. As in the previous case, we will discard linear terms, which will never exceed $19n<\tfrac{1}{2000}n^2$.

Subcase 1: $\ttt_1=2k-\tfrac{n}{3}$ and $\ttt_2=0$; $k\le\tfrac{n}{4}$.

We have
\begin{align}
\label{eq:buH}
h(2k-\tfrac{n}{3},0,0,\tfrac{n-3k}{3},\mmm,\iii)&=\tfrac{7kn}{3}-\tfrac{n^2}{18}-3k^2-k+\tfrac{n}{6}+2\\
\nonumber
&\approx \tfrac{7kn}{3}-\tfrac{n^2}{18}-3k^2\,, \text{and}\\
\label{eq:buHH}
h(2k-\tfrac{n}{3},0,0,\tfrac{n-3k}{3},\mmm,\iii)-e(E_2(n,k))&\approx \tfrac{7kn}{3}-\tfrac{11n^2}{36}-5k^2\,,
\end{align}
which is maximised at $k=\tfrac{7n}{30}$ where we obtain $\tfrac{-n^2}{45}$.

Subcase 2: $\ttt_1=2k-\tfrac{n}{3}$ and $\ttt_2=0$; $k>\tfrac{n}{4}$.
We have
\begin{equation}\label{eq:bHi}
h(2k-\tfrac{n}{3},0,0,\tfrac{n-3k}{3},\mmm,\iii)-e(E_3(n,k))\approx \tfrac{kn}{3}-\tfrac{n^2}{18}-k^2\,,
\end{equation}
which function is maximised at $k=\tfrac{n}{6}$ where we obtain $\tfrac{-n^2}{36}$.

Subcase 3: $\ttt_1=0$ and $\ttt_2=2k-\tfrac{n}{3}$; $\tfrac{n}{5}\le k\le\tfrac{n}{4}$.
Since $2\mmm=n-3k-\iii$ we have
\begin{equation}\label{eq:dvPl}
h(0,2k-\tfrac{n}{3},0,\tfrac{n-3k}{3},\mmm,\iii)=\tfrac{7kn}{6}+\tfrac{n^2}{18}+\iii(k-\tfrac{n}{6})+2k-\tfrac{n}{3}+2\,,
\end{equation}
which is maximised over $\iii$ when $\iii=n-3k$. We can therefore assume $\iii=n-3k$ and obtain
\begin{equation}
\label{eq:MAXIMA10}
h(0,2k-\tfrac{n}{3},0,\tfrac{n-3k}{3},\mmm,\iii)-e(E_2(n,k))\approx \tfrac{8kn}{3}-\tfrac{13n^2}{36}-5k^2\,,
\end{equation}
which function is maximised at $k=\tfrac{4n}{15}$ where the value is $\tfrac{-n^2}{180}$.

Subcase 4: $\ttt_1=0$ and $\ttt_2=2k-\tfrac{n}{3}$;
$\tfrac{n}{4}<k\le\tfrac{3n}{10}$. As in~\eqref{eq:dvPl}, the function
$h(0,2k-\tfrac{n}{3},0,\tfrac{n-3k}{3},\mmm,\iii)$ is maximised over $\iii$
when $\iii=n-3k$ and therefore we assume this is the case. 
We have
\begin{equation}
\label{eq:CerBel}
h(0,2k-\tfrac{n}{3},0,\tfrac{n-3k}{3},\mmm,\iii)-e(E_3(n,k))\approx \tfrac{2kn}{3}-\tfrac{n^2}{9}-k^2\,,
\end{equation}
which function is maximised at $k=n/3$, so within this range of $k$ the maximum is at $k=\tfrac{3n}{10}$ where the value is $\tfrac{-n^2}{900}$.

Subcase 5: $\ttt_1=0$ and $\ttt_2=2k-\tfrac{n}{3}$;
$\tfrac{3n}{10}<k\le\tfrac{n}{3}$. As in~\eqref{eq:dvPl}, the function
$h(0,2k-\tfrac{n}{3},0,\tfrac{n-3k}{3},\mmm,\iii)$ is maximised over $\iii$
when $\iii=n-3k$ and therefore we assume this is the case. 
We have
\begin{equation}
\label{eq:BelCer}
h(0,2k-\tfrac{n}{3},0,\tfrac{n-3k}{3},\mmm,\iii)-e(E_4(n,k))\approx \tfrac{17kn}{3}-\tfrac{11n^2}{18}-12k^2\,,\end{equation}
which function is maximised at $k=\tfrac{17n}{72}$ where the value is $\tfrac{-41n^2}{540}$.

These subcases are exhaustive, and it follows that if $\ttt_4=\tfrac{n-3k}{3}$ then $g_\ell<e(n,k)$.

It remains to consider the possibility $\ttt_4>\tfrac{n-3k}{3}$. Observe that~\eqref{ht1t4l},~\eqref{ht2t4l} and~\eqref{ht3t4l} are quadratics in $x$ with positive $x^2$ coefficient. In particular, if $h$ is maximised and $3\ttt_4>n-3k$ then $\ttt_1=\ttt_2=\ttt_3=0$. The same statement is almost true replacing $h$ with $g_\ell$: the only problem is that when $\ttt_1$ is varied, the function is discontinuous, being greater by $20$ than it `should be', at $\ttt_1=1$, that being where~\eqref{defglt11} is used rather than~\eqref{defglnorm}. Nevertheless, we can conclude that if $g_\ell$ is maximised then $\ttt_2=\ttt_3=0$ and $\ttt_1\in\{0,1\}$. We have
\[g_\ell(\ttt_1,0,0,k-\ttt_1,\mmm,\iii)\le (2n+3+k-7\ttt_1)\ttt_1+p(n-3\ttt_1,n-3k)+20\,.\]
This is very close to $e(E_4(n,k))$, and if $n/5\le k\le 3n/10$ then it is smaller than $e(E_3(n,k))$ by at least $\tfrac{n^2}{100}-30n>0$. If on the other hand $k>\tfrac{3n}{10}$, then $n-3\ttt_1>9(n-3k)$ and by~\eqref{eq:defpprime} we have
\begin{align}
\nonumber
&g_\ell(\ttt_1,0,0,k-\ttt_1,\mmm,\iii)=\\
\label{eq:RHSdi}
&=(2n-8k-\tfrac32-\tfrac{5\ttt_1}{2})\ttt_1-\tfrac{3n}{2}+9k^2+9k-3kn+\tfrac{n^2}{2}+\begin{cases}2&\ttt_1=0\\22&\ttt_1=1\end{cases}\,.
\end{align}
Since $2n-8k-\tfrac32-\tfrac{5\ttt_1}{2}< \tfrac{-2n}{5}<-20$ we conclude that $g_\ell$ is maximised with $\ttt_1=0$. Now we have
\begin{equation}\label{eq:AhojPepo}
g_\ell(0,0,0,k,\mmm,\iii)=p(n,n-3k)=e(E_4(n,k))\,,
\end{equation}
as desired.

{\sl Case 3:} $3\ttt_4<\max(528,3\kappa_0)$ and $n-3k<\max(528,3\kappa_0)$.

Observe that, taking the largest terms of each of~\eqref{defglsmt4},~\eqref{defglnorm} and~\eqref{defglt11}, we have
\begin{equation*}\begin{split}
g_\ell(\ttt_1,\ttt_2,&\ttt_3,\ttt_4,\mmm,\iii) \\
&\le 8\binom{k-\ttt_4}{2}+(8\ttt_4+3)(k-\ttt_4)+\binom{3\ttt_4}{2}+(n-3k)n+20\\
&\le \tfrac{4n^2}{9}+2\max(528,3\kappa_0)n+\binom{\max(528,3\kappa_0)}{2}+20\\
&\le e(E_4(n,k))-\tfrac{n^2}{18}+8\max(528,3\kappa_0)n+\max(528,3\kappa_0)^2\\
&<e(E_4(n,k))\,,
\end{split}\end{equation*}
where the final inequality uses $n\ge 300\max(528,3\kappa_0)$.

These cases are exhaustive, completing the proof.
\end{proof}
\end{document}